\documentclass[12p]{amsart}
\usepackage{amssymb}
\usepackage{amsmath}
\usepackage{amsfonts}
\usepackage{geometry}
\usepackage{graphicx}
\usepackage{mathrsfs,amssymb,latexsym}

\usepackage{hyperref}
\usepackage{cleveref}

\theoremstyle{plain}

\newtheorem{definition}{Definition}

\newtheorem{lemma}{Lemma}

\newtheorem{proposition}{Proposition}
\newtheorem{remark}{Remark}

\newtheorem{theorem}{Theorem}
\numberwithin{equation}{section}

\usepackage{cancel}

\begin{document}
\title{Sharp Global well-posedness and scattering for the radial conformal nonlinear wave equation}
\date{\today}
\author{Benjamin Dodson}
\maketitle

\begin{abstract}
In this paper we prove global well-posedness and scattering for the conformal, defocusing, nonlinear wave equation with radial initial data in the critical Sobolev space, for dimensions $d \geq 4$. This result extends a previous result proving sharp scattering in the three dimensional case.
\end{abstract}

\section{Introduction}
In this paper we prove global well-posedness and scattering for the conformal wave equation
\begin{equation}\label{1.1}
u_{tt} - \Delta u + |u|^{\frac{4}{d - 1}} u = 0, \qquad u(0, x) = u_{0}, \qquad u_{t}(0, x) = u_{1}, \qquad u : \mathbb{R} \times \mathbb{R}^{d} \rightarrow \mathbb{R},
\end{equation}
with radially symmetric initial data in dimensions $d \geq 4$. This continues an earlier study we began in \cite{dodson2018globalAPDE}, \cite{dodson2018global}, \cite{dodson2018global2}, and \cite{dodson2022global}. See also \cite{miao2020global}. \medskip

Specifically, we prove a sharp scattering result for radially symmetric initial data.
\begin{theorem}\label{t1.1}
The initial value problem $(\ref{1.1})$ is globally well-posed and scattering for any radially symmetric initial data $u_{0} \in \dot{H}^{1/2}(\mathbb{R}^{d})$ and $u_{1} \in \dot{H}^{-1/2}(\mathbb{R}^{d})$. Moreover, there exists a function $C(d, \| u_{0} \|_{\dot{H}^{1/2}}, \| u_{1} \|_{\dot{H}^{-1/2}})$,
\begin{equation}\label{1.2}
C : \mathbb{Z}_{\geq 4} \times [0, \infty) \times [0, \infty) \rightarrow [0, \infty),
\end{equation}
such that if $u$ is the solution to $(\ref{1.1})$ with radially symmetric initial data in $u_{0}$, $u_{1}$,
\begin{equation}\label{1.3}
\| u \|_{L_{t,x}^{\frac{2(d + 1)}{d - 1}}(\mathbb{R} \times \mathbb{R}^{d})} \leq C(d, \| u_{0} \|_{\dot{H}^{1/2}}, \| u_{1} \|_{\dot{H}^{-1/2}}).
\end{equation}
\end{theorem}

\begin{definition}[Global well-posedness and scattering]
Here we use the standard definitions of global well-posedness and scattering. Specifically, global well-posedness means that a solution to $(\ref{1.1})$ exists, the solution is unique, and the solution depends continuously on the initial data. In this paper, a solution means a solution $u \in L_{t, loc}^{\frac{2(d + 1)}{d - 1}} L_{x}^{\frac{2(d + 1)}{d - 1}}$ which satisfies Duhamel's principle,
\begin{equation}\label{1.3.1}
u(t) = \cos(t \sqrt{-\Delta}) u_{0} + \frac{\sin(t \sqrt{-\Delta})}{\sqrt{-\Delta}} u_{1} - \int_{0}^{t} \frac{\sin((t - \tau) \sqrt{-\Delta}}{\sqrt{-\Delta}} |u(\tau)|^{\frac{4}{d - 1}} u(\tau) d\tau.
\end{equation}
By scattering, we mean that there exist $u_{0}^{+}, u_{0}^{-} \in \dot{H}^{1/2}$, $u_{1}^{+}, u_{1}^{-} \in \dot{H}^{-1/2}$, such that,
\begin{equation}\label{1.3.2}
\lim_{t \rightarrow \infty} \| u(t) - \cos(t \sqrt{-\Delta}) u_{0}^{+} - \frac{\sin(t \sqrt{-\Delta})}{\sqrt{-\Delta}} u_{1}^{+} \|_{\dot{H}^{1/2}} = 0,
\end{equation}
\begin{equation}
\lim_{t \rightarrow \infty} \| \partial_{t}(u(t) - \cos(t \sqrt{-\Delta}) u_{0}^{+} - \frac{\sin(t \sqrt{-\Delta})}{\sqrt{-\Delta}} u_{1}^{+}) \|_{\dot{H}^{-1/2}} = 0,
\end{equation}
\begin{equation}\label{1.3.3}
\lim_{t \rightarrow -\infty} \| u(t) - \cos(t \sqrt{-\Delta}) u_{0}^{-} - \frac{\sin(t \sqrt{-\Delta})}{\sqrt{-\Delta}} u_{1}^{-} \|_{\dot{H}^{1/2}} = 0,
\end{equation}
and
\begin{equation}\label{1.3.4}
\lim_{t \rightarrow -\infty} \| \partial_{t}(u(t) - \cos(t \sqrt{-\Delta}) u_{0}^{-} - \frac{\sin(t \sqrt{-\Delta})}{\sqrt{-\Delta}} u_{1}^{-}) \|_{\dot{H}^{-1/2}} = 0.
\end{equation}
\end{definition}

Theorem $\ref{t1.1}$ is sharp due to the scaling symmetry. Specifically, if $u$ solves $(\ref{1.1})$, then for any $\lambda > 0$,
\begin{equation}\label{1.4}
v(t, x) = \lambda^{\frac{d - 1}{2}} u(\lambda t, \lambda x),
\end{equation}
also solves $(\ref{1.1})$ with initial data
\begin{equation}\label{1.5}
v(0, x) = \lambda^{\frac{d - 1}{2}} u_{0}(\lambda x), \qquad v_{t}(0, x) = \lambda^{\frac{d + 1}{2}} u_{1}(\lambda x).
\end{equation}
Note that $\| v(0, x) \|_{\dot{H}^{1/2}} = \| u_{0} \|_{\dot{H}^{1/2}}$ and $\| v_{t}(0, x) \|_{\dot{H}^{-1/2}} = \| u_{1} \|_{\dot{H}^{-1/2}}$ for any $\lambda > 0$. This fact was well-exploited by \cite{lindblad1995existence} to prove ill-posedness for initial data in $\dot{H}^{s} \times \dot{H}^{s - 1}$ for $s < \frac{1}{2}$. See also \cite{christ2003asymptotics}.

\subsection{Outline of previous results}
Previous interest in the conformal wave equation, $(\ref{1.1})$, has mainly focused on the $d = 3$ case. In this case, $(\ref{1.1})$ is the cubic wave equation,
\begin{equation}\label{1.6}
u_{tt} - \Delta u + u^{3} = 0.
\end{equation}
It has been known for a long time that global well-posedness and scattering hold for initial data in $\dot{H}^{1} \times L^{2}$ that decays sufficiently fast as $|x| \rightarrow \infty$, see \cite{strauss1981nonlinear} and \cite{strauss1968decay}. Observe that such initial data has the conserved energy,
\begin{equation}\label{1.7}
E(u) = \frac{1}{2} \int |\nabla u(t, x)|^{2} dx + \frac{1}{2} \int |u_{t}(t, x)|^{2} dx + \frac{d - 1}{2(d + 1)} \int |u(t,x)|^{\frac{2(d + 1)}{d - 1}} dx.
\end{equation}
Conservation of the conformal energy gives scattering, which will be shown in section three.\medskip

For initial data in $\dot{H}^{1} \cap \dot{H}^{1/2} \times L^{2} \cap \dot{H}^{-1/2}$, global well-posedness follows easily from $(\ref{1.7})$ and the local well-posedness result of \cite{lindblad1995existence}.
\begin{theorem}\label{t1.2}
The initial value problem $(\ref{1.1})$ is locally well-posed on some interval $(-T, T)$ for any $(u_{0}, u_{1}) \in \dot{H}^{1/2} \times \dot{H}^{-1/2}$, where $T = T(u_{0}, u_{1})$. Global well-posedness and scattering hold for small initial data.

Moreover, the solution satisfies
\begin{equation}\label{1.8}
u \in L_{t}^{\infty} \dot{H}^{1/2}((-T, T) \times \mathbb{R}^{d}), \qquad u_{t} \in L_{t}^{\infty} \dot{H}^{-1/2}((-T, T) \times \mathbb{R}^{d}), \qquad u \in L_{t, loc}^{\frac{2(d + 1)}{d - 1}} L_{x}^{\frac{2(d + 1)}{d - 1}}((-T, T) \times \mathbb{R}^{d}).
\end{equation}

Furthermore, if the solution only exists on an interval $[0, T_{+})$ for some $T_{+} < \infty$, then
\begin{equation}\label{1.9}
\lim_{T \nearrow T_{+}} \| u \|_{L_{t,x}^{\frac{2(d + 1)}{d - 1}}([0, T] \times \mathbb{R}^{d})} = \infty.
\end{equation}
By time reversal symmetry, the analogous result holds on $(-T_{-}, 0]$.
\end{theorem}
Therefore, $(\ref{1.1})$ has a local solution, and by $(\ref{1.7})$, blowup cannot occur in finite time.\medskip

Using the Fourier truncation method, \cite{kenig2000global} proved global well-posedness for $(\ref{1.6})$ for initial data in $H^{s} \times H^{s - 1}$ for any $s > \frac{3}{4}$. This method, introduced by \cite{bourgain1998refinements} for the nonlinear Schr{\"o}dinger equation, utilizes the smoothing effect of the Duhamel term
\begin{equation}\label{1.10}
\int_{0}^{t} \frac{\sin((t - \tau) \sqrt{-\Delta})}{\sqrt{-\Delta}} u(\tau)^{3} d\tau.
\end{equation}
The data was then split into a high frequency piece that was small and a low frequency piece that is in $\dot{H}^{1} \times L^{2}$. Global well-posedness holds for $(\ref{1.6})$ with either piece as the initial data (from Theorem $\ref{t1.2}$ and $(\ref{1.7})$). For $s > \frac{3}{4}$, it is possible to ``paste" the two solutions together and obtain a solution to $(\ref{1.6})$ for initial data in $H^{s} \times H^{s - 1}$.

This work was subsequently extended by many authors. See \cite{gallagher2003global}, \cite{bahouri2006global}, \cite{roy2007adapted}, \cite{roy2007global}, \cite{dodson2018global1}, \cite{dodson2019global} for subsequent improvements on this result.\medskip

A second approach that has proven to be very fruitful is the study of type two blowup. There are two different ways in which scattering can fail. The first way is if the $\dot{H}^{1/2}$ norm is unbounded. Since the solution to the linear wave equation is a unitary operator, it is clear that one of $(\ref{1.3.1})$--$(\ref{1.3.4})$ would fail. This is called ``type one blowup".

Scattering could also fail and yet the $\dot{H}^{1/2} \times \dot{H}^{-1/2}$ norm remain bounded. This behavior is frequently observed for a soliton, although see also the pseudoconformal transformation of the soliton for the nonlinear Schr{\"o}dinger equation (see for example \cite{merle1993determination}).

In \cite{dodson2015scattering} we proved global well-posedness and scattering for radially symmetric solutions to $(\ref{1.6})$ with $\| u(t) \|_{\dot{H}^{1/2}} + \| u_{t}(t) \|_{\dot{H}^{-1/2}}$ uniformly bounded on the entire time of its existence. See also  \cite{shen2014energy}, \cite{dodson2015scattering1}, and \cite{dodson2020scattering}. The proof in \cite{dodson2015scattering} uses the concentration compactness argument, excluding the existence of a non-scattering solution of minimal size.\medskip

It is worth noting that the result in \cite{dodson2015scattering} (and in  \cite{shen2014energy}, \cite{dodson2015scattering1}) holds for both the defocusing and the focusing case. This is because, unlike in the energy-critical case, there does not exist a soliton solution to $(\ref{1.6})$ that lies in $\dot{H}^{1/2} \times \dot{H}^{-1/2}$.

Still, for the focusing, cubic wave equation,
\begin{equation}\label{1.11}
u_{tt} - \Delta u - u^{3} = 0,
\end{equation}
there do exist solutions for which the $\dot{H}^{1/2} \times \dot{H}^{-1/2}$ is unbounded. Indeed, for this equation, the energy is given by
\begin{equation}\label{1.12}
E(u) = \frac{1}{2} \int |\nabla u(t,x)|^{2} dx + \frac{1}{2} \int |u_{t}(t, x)|^{2} dx - \frac{d - 1}{2(d + 1)} \int |u(t,x)|^{\frac{2(d + 1)}{d - 1}} dx,
\end{equation}
which unlike $(\ref{1.7})$ does not prevent any norms of a solution to $(\ref{1.11})$ from getting arbitrarily large. In fact, it is well known that there exist solutions to $(\ref{1.1})$ for which the $\dot{H}^{1/2} \times \dot{H}^{-1/2}$ norm is unbounded. See \cite{duyckaerts2023global} for the state of the art in this direction and a description of prior results.

\subsection{Outline of the argument}
Theorem $\ref{t1.1}$ is proved using the Fourier truncation method and conservation of the conformal energy. Specifically, inspired by \cite{kenig2000global}, we split the initial data into two pieces,
\begin{equation}\label{1.13}
u_{0} = v_{0} + w_{0}, \qquad u_{1} = v_{1} + w_{1},
\end{equation}
where $(v_{0}, v_{1})$ has finite conformal energy and $(w_{0}, w_{1})$ has small $\dot{H}^{1/2} \times \dot{H}^{-1/2}$ norm.

Following \cite{strauss1981nonlinear} and \cite{lindblad1995existence}, we know that $(\ref{1.1})$ is scattering with initial data $(v_{0}, v_{1})$ and $(w_{0}, w_{1})$. Therefore, it remains to handle the cross-terms. The contribution of the cross terms $F$ is of the form
\begin{equation}\label{1.14}
\langle (t + |x|) Lv + (d - 1) v, (t + |x|) F \rangle, \qquad L = \partial_{t} + \frac{x}{|x|} \cdot \nabla,
\end{equation}
where $F$ contains terms of the form $|w|^{\frac{4}{d - 1}} |v|$. In \cite{dodson2022global},
\begin{equation}\label{1.15}
(\ref{1.14}) \lesssim \| (t + |x|) Lv + (d - 1) v \|_{L^{2}} \| (t + |x|) |w|^{\frac{2}{d - 1}} \|_{L^{\infty}} \| w \|_{L^{\frac{2(d + 1)}{d - 1}}}^{\frac{d - 3}{d - 1}} \| v ||_{L^{\frac{2(d + 1)}{d - 1}}}^{\frac{4}{d - 1}}.
\end{equation}
Combining the radially symmetric Sobolev embedding theorem and the dispersive estimate,
\begin{equation}\label{1.16}
\| (t + |x|) |w|^{\frac{2}{d - 1}} \|_{L^{\infty}} \lesssim 1,
\end{equation}
and therefore $(\ref{1.15})$ implies a bound on the integral of $\frac{\mathcal E(t)}{t^{2}}$, where $\mathcal E(t)$ is the conformal energy.

For initial data in $\dot{H}^{1/2} \times \dot{H}^{-1/2}$, we still have the bound
\begin{equation}\label{1.17}
\| |x| |w|^{\frac{2}{d - 1}} \|_{L^{\infty}} \lesssim 1,
\end{equation}
from the radial Sobolev embedding theorem,
along with the bound
\begin{equation}\label{1.17.1}
\| t |w|^{\frac{2}{d - 1}} \|_{L^{\infty}(|x| \geq \delta |t|)} \lesssim \frac{1}{\delta}.
\end{equation}
\begin{remark}
For the discussion in this section, it is reasonable to ignore the logarithmic divergence of the radial Sobolev embedding that arises from the Littlewood--Paley projection in $L^{\infty}$.
\end{remark}

On the other hand, for general initial data in $\dot{H}^{1/2} \times \dot{H}^{-1/2}$, there is no reason to think that the dispersive estimate $\| t |w|^{\frac{2}{d - 1}} \|_{L^{\infty}}$ will hold, since the $\dot{H}^{1/2} \times \dot{H}^{-1/2}$ norm is invariant under the operator
\begin{equation}\label{1.18}
\begin{pmatrix}
\cos(t \sqrt{-\Delta}) & \frac{\sin(t \sqrt{-\Delta})}{\sqrt{-\Delta}} \\ -\sqrt{-\Delta} \sin(t \sqrt{-\Delta}) & \cos(t \sqrt{-\Delta})
\end{pmatrix}.
\end{equation}
Instead, we use that the square $L^{2}$ norm of $\nabla_{t,x} v$ is bounded by the conformal energy divided by $\frac{1}{t^{2}}$, which gives us good decay to cancel out the contribution of $t |w|^{\frac{4}{d - 1}} |v|$. We use the Morawetz estimate and the local energy decay to do this, which gives a bound on the scattering size.\medskip

In section two, we recall some Strichartz estimates and the small data result of \cite{lindblad1995existence}. In section three we recall the scattering result of \cite{strauss1981nonlinear}. In section four, we prove scattering in the $d = 4$ case. In section five, we prove a modified small data result in dimensions $d > 5$. In section six, we prove scattering in the $d > 4$ case. Finally, in section seven we complete the proof of Theorem $\ref{t1.1}$ using the profile decomposition.

\section{Strichartz estimates and small data results}
Global well-posedness and scattering for $(\ref{1.1})$ with small initial data is a direct result of Strichartz estimates.
\begin{theorem}[Strichartz estimates]\label{t4.1}
Let $I$ be a time interval and let $u : I \times \mathbb{R}^{d} \rightarrow \mathbb{C}$ be a Schwartz solution to the wave equation,
\begin{equation}\label{4.1}
u_{tt} - \Delta u + F = 0, \qquad u(t_{0}, \cdot) = u_{0}, \qquad u_{t}(t_{0}, \cdot) = u_{1}, \qquad \text{for some} \qquad t_{0} \in I.
\end{equation}
Then for $s \geq 0$, $2 \leq p, \tilde{p} \leq \infty$, $2 \leq q, \tilde{q} < \infty$ obeying the scaling conditions
\begin{equation}\label{4.2}
\frac{1}{p} + \frac{d}{q} = \frac{d}{2} - s = \frac{1}{\tilde{p}'} + \frac{d}{\tilde{q}'} - 2,
\end{equation}
and the wave admissibility conditions
\begin{equation}\label{4.3}
\frac{1}{p} + \frac{d - 1}{2q}, \frac{1}{\tilde{p}} + \frac{d - 1}{2 \tilde{r}} \qquad \leq \qquad \frac{d - 1}{4},
\end{equation}
\begin{equation}\label{4.4}
\| u \|_{L_{t}^{p} L_{x}^{q}(I \times \mathbb{R}^{d})} + \| u \|_{L_{t}^{\infty} \dot{H}_{x}^{s}(I \times \mathbb{R}^{d})} + \| \partial_{t} u \|_{L_{t}^{\infty} \dot{H}_{x}^{s - 1}(I \times \mathbb{R}^{d})} \lesssim_{s, p, q, \tilde{p}, \tilde{q}, s, d} \| u_{0} \|_{\dot{H}_{x}^{s}(\mathbb{R}^{d})} + \| u_{1} \|_{\dot{H}_{x}^{s - 1}(\mathbb{R}^{d})} + \| F \|_{L_{t}^{\tilde{p}'} L_{x}^{\tilde{q}'}(I \times \mathbb{R}^{d})}.
\end{equation}
\end{theorem}
\begin{proof}
This theorem was copied from \cite{tao2006nonlinear}. See \cite{kato1994lq}, \cite{ginibre1995generalized}, \cite{kapitanski1989some}, \cite{lindblad1995existence}, \cite{sogge1995lectures}, \cite{shatah1993regularity}, and \cite{keel1998endpoint} for references.
\end{proof}

Of particular importance to this paper is the conformal Strichartz estimate,
\begin{equation}\label{4.5}
\| u \|_{L_{t,x}^{\frac{2(d + 1)}{d - 1}}(\mathbb{R} \times \mathbb{R}^{d})} \lesssim_{d} \| u(0) \|_{\dot{H}_{x}^{1/2}(\mathbb{R}^{d})} + \| u_{t}(0) \|_{\dot{H}_{x}^{-1/2}(\mathbb{R}^{d})} + \| F \|_{L_{t, x}^{\frac{2(d + 1)}{d + 3}}(\mathbb{R} \times \mathbb{R}^{d})},
\end{equation}
which was proved in the original paper \cite{strichartz1977restrictions}. A straightforward application of $(\ref{4.5})$ gives global well-posedness and scattering for $(\ref{1.1})$ with small initial data, for both radially symmetric initial data and general initial data, see \cite{lindblad1995existence}.

\begin{theorem}\label{t4.2}
For any $d > 3$, there exists some $\epsilon_{0}(d) > 0$ such that if
\begin{equation}\label{4.6}
\| u(0, \cdot) \|_{\dot{H}^{1/2}(\mathbb{R}^{d})} + \| u_{t}(0, \cdot) \|_{\dot{H}^{-1/2}(\mathbb{R}^{d})} \leq \epsilon_{0}(d),
\end{equation}
then $(\ref{4.1})$ is globally well-posed and the solution satisfies
\begin{equation}\label{4.7}
\| u \|_{L_{t,x}^{\frac{2(d + 1)}{d - 1}}(\mathbb{R} \times \mathbb{R}^{d})} \lesssim \| u(0, \cdot) \|_{\dot{H}^{1/2}(\mathbb{R}^{d})} + \| u_{t}(0, \cdot) \|_{\dot{H}^{-1/2}(\mathbb{R}^{d})}.
\end{equation}
Moreover, if $u$ solves $(\ref{1.1})$, $\| u \|_{L_{t,x}^{\frac{2(d + 1)}{d - 1}}(\mathbb{R} \times \mathbb{R}^{d})} < \infty$ is equivalent to scattering to a free solution both forward and backward in time.
\end{theorem} 
\begin{remark}
This theorem is proved in many places and in far more generality, see for example \cite{tao2006nonlinear}. Still, for the large data result in dimensions $d \geq 5$, it will be useful to prove a slight modification of the small data result. Because of this, it is instructive to give a short proof of the small data result here.
\end{remark}
\begin{proof}
The theorem is proved using Picard iteration. Define the sequence
\begin{equation}\label{4.8}
u^{(0)}(t) = \cos(t \sqrt{-\Delta}) u_{0} + \frac{\sin(t \sqrt{-\Delta})}{\sqrt{-\Delta}} u_{1},
\end{equation}
and for $n \geq 1$,
\begin{equation}\label{4.9}
u^{(n)}(t) = u^{(0)}(t) - \int_{0}^{t} \frac{\sin((t - \tau) \sqrt{-\Delta})}{\sqrt{-\Delta}} |u^{(n - 1)}(\tau)|^{\frac{4}{d - 1}} u^{(n - 1)}(\tau) d\tau.
\end{equation}
Then, by $(\ref{4.5})$,
\begin{equation}\label{4.10}
\| u^{(n)}(t) \|_{L_{t,x}^{\frac{2(d + 1)}{d - 1}}(\mathbb{R} \times \mathbb{R}^{d})} \lesssim \| u_{0} \|_{\dot{H}^{1/2}} + \| u_{1} \|_{\dot{H}^{-1/2}} + \| u^{(n - 1)} \|_{L_{t,x}^{\frac{2(d + 1)}{d - 1}}(\mathbb{R} \times \mathbb{R}^{d})}^{1 + \frac{4}{d - 1}}.
\end{equation}
Therefore, for $\epsilon_{0}(d) > 0$ sufficiently small and some constant $C(d)$ sufficiently large,
\begin{equation}\label{4.11}
\| u^{(n - 1)} \|_{L_{t,x}^{\frac{2(d + 1)}{d - 1}}(\mathbb{R} \times \mathbb{R}^{d})} \leq C(d) \epsilon_{0} \hspace{5mm} \Rightarrow \hspace{5mm} \| u^{(n)} \|_{L_{t,x}^{\frac{2(d + 1)}{d - 1}}(\mathbb{R} \times \mathbb{R}^{d})} \leq C(d) \epsilon_{0},
\end{equation}
and therefore, by induction,
\begin{equation}\label{4.11}
\| u^{(n)} \|_{L_{t,x}^{\frac{2(d + 1)}{d - 1}}(\mathbb{R} \times \mathbb{R}^{d})} \leq C(d) \epsilon_{0}, \qquad \forall n.
\end{equation}
Also, by $(\ref{4.5})$ and $(\ref{4.11})$,
\begin{equation}\label{4.12}
\| u^{(n)} - u^{(n - 1)} \|_{L_{t, x}^{\frac{2(d + 1)}{d - 1}}(\mathbb{R} \times \mathbb{R}^{d})} \lesssim [C(d) \epsilon_{0}]^{\frac{4}{d - 1}} \| u^{(n - 1)} - u^{(n - 2)} \|_{L_{t, x}^{\frac{2(d + 1)}{d - 1}}(\mathbb{R} \times \mathbb{R}^{d})},
\end{equation}
which by the contraction mapping theorem proves that $u^{(n)}(t)$ converges in $L_{t,x}^{\frac{2(d + 1)}{d - 1}}(\mathbb{R} \times \mathbb{R}^{d})$ to a unique solution.
\end{proof}

While global well-posedness and scattering hold for small nonradial data, the proof in this paper of global well-posedness and scattering for $(\ref{1.1})$ with large initial data relies heavily on radial symmetry. In particular, the proof relies heavily on the radial Strichartz estimate of \cite{sterbenz2005angular}. See also \cite{rodnianskisterbenz}.
\begin{theorem}[Strichartz estimates for radially symmetric initial data]\label{t4.3}
Let $u$ be a radially symmetric function on $\mathbb{R}^{d + 1}$ such that $u_{tt} - \Delta u = 0$. Then, the following estimates hold,
\begin{equation}\label{4.13}
\| u \|_{L_{t}^{p} L_{x}^{q}(\mathbb{R} \times \mathbb{R}^{d})} \lesssim \| u(0) \|_{\dot{H}^{\gamma}} + \| u_{t}(0) \|_{\dot{H}^{\gamma - 1}},
\end{equation}
where 
\begin{equation}\label{4.14}
\frac{1}{p} + \frac{d - 1}{q} < \frac{d - 1}{2}, \qquad \text{and} \qquad \frac{1}{p} + \frac{d}{q} = \frac{d}{2} - \gamma.
\end{equation}
\end{theorem}

Observe that after doing some algebra with $(\ref{4.14})$, if $p = 2$,
\begin{equation}\label{4.15}
\frac{1}{q} < \frac{1}{2} \cdot \frac{d - 2}{d - 1}.
\end{equation}
Note that one particular case of $(\ref{4.13})$ is
\begin{equation}\label{4.16}
\| u \|_{L_{t}^{2} L_{x}^{\frac{2d}{d - 2}}(\mathbb{R} \times \mathbb{R}^{d})} \lesssim \| u(0) \|_{\dot{H}^{1/2}} + \| u_{t}(0) \|_{\dot{H}^{-1/2}}.
\end{equation}
Moreover, we combining Theorem $\ref{t4.3}$ with the radial Sobolev embedding theorem implies
\begin{lemma}\label{l4.4}
For any $0 < \theta \leq 1$, $d \geq 3$,
\begin{equation}\label{4.17}
\| |x|^{\frac{d - 2}{2}(1 - \theta)}u \|_{L_{t}^{2} L_{x}^{\frac{2d}{(d - 2) \theta}}(\mathbb{R} \times \mathbb{R}^{d})} \lesssim_{d, \theta} \| u(0) \|_{\dot{H}^{1/2}} + \| u_{t}(0) \|_{\dot{H}^{-1/2}}.
\end{equation}
\end{lemma}
\begin{proof}
Choose some $q$ very close to $\frac{1}{2} \cdot \frac{d - 2}{d - 1}$. For $\gamma = \frac{d - 1}{2} - \frac{d}{q}$, Theorem $\ref{t4.3}$ implies that
\begin{equation}\label{4.18}
\| |\nabla|^{\gamma} u \|_{L_{t}^{2} L_{x}^{q}(\mathbb{R} \times \mathbb{R}^{d})} \lesssim \| u(0) \|_{\dot{H}^{1/2}} + \| u_{t}(0) \|_{\dot{H}^{-1/2}}.
\end{equation}
Then, by the radial Sobolev embedding theorem,
\begin{equation}\label{4.19}
\| |x|^{\frac{d - 1}{2} \gamma} u \|_{L_{t}^{2} L_{x}^{r}(\mathbb{R} \times \mathbb{R}^{d})} \lesssim \| u(0) \|_{\dot{H}^{1/2}} + \| u_{t}(0) \|_{\dot{H}^{-1/2}}, \qquad \frac{1}{r} = \frac{1}{q} - \gamma.
\end{equation}
Interpolating $(\ref{4.16})$ and $(\ref{4.19})$ gives $(\ref{4.17})$.
\end{proof}

The Christ--Kiselev lemma implies that Theorem $\ref{t4.3}$ and Lemma $\ref{l4.4}$ also hold for a small data solution to $(\ref{1.1})$.

\begin{lemma}[Christ--Kiselev lemma]\label{l4.5}
Let $X, Y$ be Banach spaces, let $I$ be a time interval, and let $K \in C^{0}(I \times I \rightarrow B(X \rightarrow Y))$ be a kernel taking values in the space of bounded operators from $X$ to $Y$. If $1 \leq p < q \leq \infty$ is such that
\begin{equation}\label{4.20}
\| \int_{I} K(t, s) f(s) ds \|_{L_{t}^{q}(I \rightarrow Y)} \leq A \| f \|_{L_{t}^{p}(I \rightarrow X)},
\end{equation}
for all $f \in L_{t}^{p}(I \rightarrow X)$ and some $A > 0$, then we also have
\begin{equation}\label{4.21}
\| \int_{s \in I : s < t} K(t, s) f(s) ds \|_{L_{t}^{q}(I \rightarrow Y)} \lesssim_{p, q} A \| f \|_{L_{t}^{p}(I \rightarrow X)}.
\end{equation}
\end{lemma}
\begin{proof}
This lemma was copied out of \cite{tao2006nonlinear}. This lemma was proved in \cite{christ2001maximal}. See also \cite{smith2000global} or \cite{tao2000spherically}.

\end{proof}

\section{Conformal energy and Morawetz estimates}
For large initial data, global well-posedness and scattering for $(\ref{1.1})$ is equivalent to proving that $(\ref{1.1})$ has a solution which satisfies
\begin{equation}\label{5.1}
\| u \|_{L_{t,x}^{\frac{2(d + 1)}{d - 1}}(\mathbb{R} \times \mathbb{R}^{d})} < \infty.
\end{equation}
Indeed, if $(\ref{5.1})$ holds, $\mathbb{R}$ can be partitioned into finitely many subintervals $I_{j}$ for which
\begin{equation}\label{5.2}
\| u \|_{L_{t,x}^{\frac{2(d + 1)}{d - 1}}(I_{j} \times \mathbb{R}^{d})} \ll 1.
\end{equation}
One can then use the Picard iteration argument from Theorem $\ref{t4.2}$ to prove global well-posedness and scattering.\medskip

On the other hand, if scattering is known to occur, then by $(\ref{4.5})$ and the Picard iteration argument from Theorem $\ref{t4.2}$, $(\ref{5.1})$ holds.\medskip

For large data, \cite{strauss1981nonlinear} and \cite{strauss1968decay} proved global well-posedness and scattering for $(\ref{1.1})$ with large initial data with sufficient regularity and decay.
\begin{theorem}\label{t5.1}
Suppose $u_{0}$ and $u_{1}$ are initial data that satisfy
\begin{equation}\label{5.3}
\| \langle x \rangle \nabla u_{0} \|_{L^{2}} + \| u_{0} \|_{L^{2}} + \| \langle x \rangle u_{1} \|_{L^{2}} + \| \langle x \rangle^{\frac{d - 1}{d + 1}} u_{0} \|_{L_{x}^{\frac{2(d + 1)}{d - 1}}} < \infty.
\end{equation}
Here, $\langle x \rangle = (1 + |x|^{2})^{1/2}$. Then the solution to $(\ref{1.1})$ is globally well-posed and scattering.
\end{theorem}
\begin{proof}
The conformal energy,
\begin{equation}\label{5.4}
\aligned
\mathcal E(u) = \frac{1}{4} \int_{\mathbb{R}^{d}} |(t + |x|) Lu + (d - 1) u|^{2} + |(t - |x|) \underline{L} u + (d - 1) u|^{2} dx \\ + \frac{1}{2} \int (t^{2} + |x|^{2}) |\cancel{\nabla} u|^{2} dx + \frac{d - 1}{4(d + 1)} \int (t^{2} + |x|^{2}) |u|^{\frac{2(d + 1)}{d - 1}} dx,
\endaligned
\end{equation}
is a conserved quantity, where $L = (\partial_{t} + \frac{x}{|x|} \cdot \nabla)$ and $\underline{L} = (\partial_{t} - \frac{x}{|x|} \cdot \nabla)$.\medskip

Indeed, define the tensors
\begin{equation}\label{5.5}
\aligned
T^{00}(t,x) &= \frac{1}{2} |\partial_{t} u|^{2} + \frac{1}{2} |\nabla u|^{2} + \frac{d - 1}{2(d + 1)} |u|^{\frac{2(d + 1)}{d - 1}}, \\
T^{0j}(t,x) &= T^{j0}(t,x) = -(\partial_{t} u)(\partial_{x_{j}} u), \\
T^{jk}(t,x) &= (\partial_{x_{j}} u \partial_{x_{k}} u) - \frac{\delta_{jk}}{2}(|\nabla u|^{2} - |\partial_{t} u|^{2}) - \delta_{jk} \frac{d - 1}{2(d + 1)} |u|^{\frac{2(d + 1)}{d - 1}}.
\endaligned
\end{equation}
The tensor functions satisfy the differential equations
\begin{equation}\label{5.6}
\partial_{t} T^{00}(t,x) + \partial_{x_{j}} T^{0j}(t,x) = 0, \qquad \partial_{t} T^{0j}(t,x) + \partial_{x_{k}} T^{jk}(t,x) = 0.
\end{equation}
The Einstein summation convention is observed. The differential equations $(\ref{5.6})$ imply that the quantity
\begin{equation}\label{5.7}
Q(t) = \int (t^{2} + |x|^{2}) T^{00}(t,x) - 2t x_{j} T^{0j}(t,x) + (d - 1) t u(\partial_{t} u) - \frac{d - 1}{2} |u|^{2} dx,
\end{equation}
is conserved. Indeed, by $(\ref{5.6})$,
\begin{equation}\label{5.8}
\aligned
\frac{d}{dt} Q(t) = 2t \int T^{00}(t,x) dx - \int (t^{2} + |x|^{2}) \partial_{x_{j}} T^{0j}(t,x) dx - 2t \int x_{j} T^{0j}(t,x) dx + 2t \int x_{j} \partial_{x_{k}} T^{jk}(t,x) dx \\
+ (d - 1) \int u (\partial_{t} u) dx + (d - 1) t \int (\partial_{t} u)^{2} dx + (d - 1) t \int u(\Delta u - |u|^{\frac{4}{d - 1}} u) dx - (d - 1) \int u (\partial_{t} u) dx.
\endaligned
\end{equation}
Integrating the second term in $(\ref{5.8})$ by parts,
\begin{equation}\label{5.9}
\aligned
=  2t \int T^{00}(t,x) dx - 2t \int \delta_{jk} T^{jk}(t,x) dx  + (d - 1) t \int (\partial_{t} u)^{2} dx + (d - 1) t \int u(\Delta u - |u|^{\frac{4}{d - 1}} u) dx.
\endaligned
\end{equation}
Since $\delta_{jk} \delta_{jk} = d$,
\begin{equation}\label{5.10}
\aligned
= 2t \int T^{00}(t,x) dx - 2t \int |\nabla u|^{2} + dt \int (|\nabla u|^{2} - |\partial_{t} u|^{2}) dx + \frac{d(d - 1)t}{d + 1} \int |u|^{\frac{2(d + 1)}{d - 1}} dx \\ + (d - 1)t \int |\partial_{t} u|^{2} dx - (d - 1)t \int |\nabla u|^{2} dx - (d - 1)t \int |u|^{\frac{2(d + 1)}{d - 1}} dx.
\endaligned
\end{equation}
Doing some algebra,
\begin{equation}\label{5.11}
= 2t \int T^{00}(t,x) dx - t \int |\nabla u|^{2} dx - t \int |\partial_{t} u|^{2} dx - \frac{d - 1}{d + 1} t \int |u|^{\frac{2(d + 1)}{d - 1}} dx = 0.
\end{equation}
Therefore, $Q(t)$ is conserved.

Now then,
\begin{equation}\label{5.12}
\aligned
\int (t^{2} + |x|^{2}) T^{00}(t,x) dx - \int 2t x_{j} T^{0j}(t,x) dx \\ = \int (t^{2} + |x|^{2}) (\frac{1}{2} |\partial_{t} u|^{2} + \frac{1}{2} |\frac{x}{|x|} \cdot \nabla u|^{2} + \frac{1}{2} |\cancel{\nabla} u|^{2} + \frac{d - 1}{2(d + 1)} |u|^{\frac{2(d + 1)}{d - 1}}) dx \\
= \frac{1}{4} \int (t + |x|)^{2} |Lu|^{2} dx + \frac{1}{4} \int (t - |x|)^{2} |\underline{L}u|^{2} dx \\ + \frac{1}{2} \int (t^{2} + |x|^{2}) |\cancel{\nabla} u|^{2} dx + \frac{d - 1}{2(d + 1)} \int (t^{2} + |x|^{2}) |u|^{\frac{2(d + 1)}{d - 1}} dx.
\endaligned
\end{equation}
Next, integrating by parts,
\begin{equation}\label{5.13}
\aligned
\frac{1}{2} \langle (t + |x|) Lu, (d - 1) u \rangle_{L^{2}} + \frac{1}{2} \langle (t - |x|) Lu, (d - 1) u \rangle_{L^{2}} \\ = (d - 1)t \int (\partial_{t} u) u dx + (d - 1) \int u (x \cdot \nabla u) dx = (d - 1)t \int (\partial_{t} u) u dx - \frac{d(d - 1)}{2} \int |u|^{2} dx.
\endaligned
\end{equation}
Since
\begin{equation}\label{5.14}
-\frac{d(d - 1)}{2} \int |u|^{2} dx + \frac{(d - 1)^{2}}{2} \int |u|^{2} dx = -\frac{d - 1}{2} \int |u|^{2} dx,
\end{equation}
$(\ref{5.12})$--$(\ref{5.14})$ imply that $Q(t)$ is equal to the right hand side of $(\ref{5.4})$.\medskip

Now then, translating in time so that the initial data is at time $t = 1$, $(\ref{5.3})$ implies that $\mathcal E(1) < \infty$. Since $\mathcal E(t)$ is a conserved quantity,
\begin{equation}\label{5.15}
\int_{1}^{\infty} \int |u|^{\frac{2(d + 1)}{d - 1}} dx dt \lesssim \int_{1}^{\infty} \frac{\mathcal E(1)}{t^{2}} dt < \infty.
\end{equation}
Time reversal symmetry of $(\ref{1.1})$ implies $(\ref{5.1})$.
\end{proof}

The computations using the stress-energy tensor also yield a Morawetz estimate.
\begin{proposition}\label{p5.2}
For any $T > 0$, if $u$ solves $(\ref{1.1})$,
\begin{equation}\label{5.16}
\aligned
\int_{0}^{T} \int [\frac{1}{|x|^{3}} u^{2} + \frac{1}{|x|} |u|^{\frac{2(d + 1)}{d - 1}}] dx dt \lesssim \sup_{t \in [0, T]} \| \nabla_{t, x} u \|_{L^{2}}^{2}.
\endaligned
\end{equation}
\end{proposition}
\begin{proof}
Since $u$ is radially symmetric, we compute in polar coordinates. Let $M(t)$ denote the Morawetz potential,
\begin{equation}\label{5.17}
M(t) = \int u_{t} u_{r} r^{d - 1} dr + \frac{d - 1}{2} \int u_{t} u r^{d - 2} dr.
\end{equation}
Using Hardy's inequality,
\begin{equation}\label{5.18}
\sup_{0 \leq t \leq T} |M(t)| \lesssim \sup_{t \in [0, T]} \| \nabla u \|_{L^{2}}^{2} + \| u_{t} \|_{L^{2}}^{2}.
\end{equation}
Next, by the product rule,
\begin{equation}\label{5.19}
\frac{d}{dt} M(t) = \int u_{t} u_{rt} r^{d - 1} dr + \frac{d - 1}{2} \int u_{t}^{2} r^{d - 2} dr + \int u_{tt} u_{r} r^{d - 1} dr + \frac{d - 1}{2} \int u_{tt} u r^{d - 2} dr.
\end{equation}
Integrating by parts, 
\begin{equation}\label{5.20}
\int u_{t} u_{rt} r^{d - 1} dr + \frac{d - 1}{2} \int u_{t}^{2} r^{d - 2} dr = 0.
\end{equation}
Next, integrating by parts, since $\Delta u = u_{rr} + \frac{d - 1}{r} u_{r}$,
\begin{equation}\label{5.21}
\aligned
\int \Delta u u_{r} r^{d - 1} dr + \frac{d - 1}{2} \int \Delta u u r^{d - 2} dr = -\frac{(d - 1)(d - 3)}{2} \int  u^{2} r^{d - 4} dr.
\endaligned
\end{equation}
Next, integrating by parts,
\begin{equation}\label{5.22}
\aligned
-\int  |u|^{\frac{4}{d - 1}} u u_{r} r^{d - 1} - \frac{d - 1}{2} \int |u|^{\frac{2(d + 1)}{d - 1}} r^{d - 2} dr = -\frac{d - 1}{d + 1} \int |u|^{\frac{2(d + 1)}{d - 1}} r^{d - 2} dr.
\endaligned
\end{equation}
Since $(\ref{5.21})$ and $(\ref{5.22})$ have the same sign, the proof is complete.
\end{proof}

Finite propagation speed also allows us to cut-off in space, which will be important to the proof of scattering. The reason for this is that examining the conformal energy in $(\ref{5.4})$ implies that
\begin{equation}\label{5.23}
t^{2} \int_{|x| \leq \frac{1}{2} |t|} |\nabla_{t,x} u(t, x)|^{2} dx \lesssim \mathcal E(t) + \int_{|x| \leq \frac{1}{2} |t|} |u(t, x)|^{2} dx.
\end{equation}
Thus, a cut-off in space yields a better bound on $(\ref{5.18})$.
\begin{proposition}\label{p5.3}
Suppose $u$ solves $(\ref{1.1})$. For any $T > 0$, if $\chi \in C_{0}^{\infty}(\mathbb{R}^{d})$, $\chi(x) = 1$ for $|x| \leq 1$, $\chi(x) = 0$ for $|x| \geq 2$, then for any $\delta > 0$,
\begin{equation}\label{5.24}
\aligned
\int_{0}^{T} \int \chi(\frac{x}{\delta T}) [\frac{1}{|x|^{3}} u^{2} + \frac{1}{|x|} |u|^{\frac{2(d + 1)}{d - 1}}] dx dt \lesssim_{\delta} \sup_{t \in [0, T]} \| \nabla_{t, x} u \|_{L^{2}(|x| \leq 2 \delta T)}^{2} + \sup_{t \in [0, T]} \frac{1}{\delta^{2} T^{2}} \| u \|_{L^{2}(|x| \leq 2 \delta T)}^{2}.
\endaligned
\end{equation}
\end{proposition}
\begin{proof}
This time use
\begin{equation}\label{5.25}
M(t) = \int \chi(\frac{x}{\delta T}) u_{t} u_{r} r^{d - 1} dr + \frac{d - 1}{2} \int \chi(\frac{x}{\delta T}) u u_{t} r^{d - 2} dr.
\end{equation}
We can use the same computations in $(\ref{5.17})$--$(\ref{5.22})$, only we also have to take into account the fact that when integrating by parts, derivatives can hit $\chi(\frac{x}{\delta T})$. Now then, since
\begin{equation}\label{5.26}
|\nabla^{(k)} \chi(\frac{x}{\delta T})| \lesssim_{k} \frac{1}{\delta^{k} T^{k}}, \qquad \text{for} \qquad k = 1, 2, 3, \qquad \text{and is supported on} \qquad \delta T \leq |x| \leq 2 \delta T.
\end{equation}
Moreover, for any $l \geq 0$,
\begin{equation}\label{5.27}
\frac{1}{r^{l}} |\nabla^{(k)} \chi(\frac{x}{\delta T})| \lesssim_{k, l} \frac{1}{\delta^{k + l} T^{k + l}}, \qquad \text{for} \qquad k = 1, 2, 3, \qquad \text{and is supported on} \qquad \delta T \leq |x| \leq 2 \delta T.
\end{equation}
Therefore, the contribution of the additional terms coming from $\chi(\frac{x}{\delta T})$ is bounded by
\begin{equation}\label{5.28}
\aligned
\frac{1}{\delta T} \int_{0}^{T} \int_{\delta T \leq |x| \leq 2 \delta T} |\nabla_{t,x} u|^{2} dx dt + \frac{1}{\delta^{3} T^{3}} \int_{0}^{T} \int_{\delta T \leq |x| \leq 2 \delta T} |u|^{2} dx dt  \\ \lesssim_{\delta} \sup_{t \in [0, T]} \| \nabla_{t, x} u \|_{L^{2}(|x| \leq 2 \delta T)}^{2} + \sup_{t \in [0, T]} \frac{1}{\delta^{2} T^{2}} \| u \|_{L^{2}(|x| \leq 2 \delta T)}^{2}.
\endaligned
\end{equation}
\end{proof}

Next, we prove a local energy decay estimate.
\begin{proposition}\label{p5.4}
For any $T > 0$, $R > 0$, if $u$ solves $(\ref{1.1})$,
\begin{equation}\label{5.29}
\aligned
R^{-1} \int_{0}^{T} \int_{|x| \leq R} \chi(\frac{x}{\delta T}) [|\nabla u|^{2} + u_{t}^{2}] dx dt \lesssim_{\delta} \sup_{t \in [0, T]} \| \nabla_{t, x} u \|_{L^{2}(|x| \leq 2 \delta T)}^{2} \\ + \sup_{t \in [0, T]} \frac{1}{\delta^{2} T^{2}} \| u \|_{L^{2}(|x| \leq 2 \delta T)}^{2} + \sup_{t \in [0, T]} \| u \|_{L^{\frac{2(d + 1)}{d - 1}}(|x| \leq 2 \delta T)}^{\frac{2(d + 1)}{d - 1}}.
\endaligned
\end{equation}
\end{proposition}
\begin{proof}
Define $\psi(r) \in C_{0}^{\infty}(\mathbb{R}^{d})$ and suppose $\psi(r) = 1$ for $0 \leq r \leq 1$, $\psi(r) = \frac{3}{2r}$ for $r > 2$, and $\partial_{r}(r \psi(r)) = \phi(r) \geq 0$ for $r \geq 0$. Now, define the Morawetz potential
\begin{equation}\label{5.30}
M(t) = R^{-1} \int \chi(\frac{r}{\delta T}) \psi(\frac{r}{R}) u_{t} u_{r} r^{d} dr + \frac{d - 1}{2} R^{-1} \int \chi(\frac{r}{\delta T}) \psi(\frac{r}{R}) u_{t} u r^{d - 1} dr.
\end{equation}
As in $(\ref{5.25})$,
\begin{equation}\label{5.31}
\sup_{0 \leq t \leq T} M(t) \lesssim_{\delta} \sup_{t \in [0, T]} \| \nabla_{t, x} u \|_{L^{2}(|x| \leq 2 \delta T)}^{2} + \sup_{t \in [0, T]} \frac{1}{\delta^{2} T^{2}} \| u \|_{L^{2}(|x| \leq 2 \delta T)}^{2}. 
\end{equation}

Next, by direct computation,
\begin{equation}\label{5.32}
\aligned
\frac{d}{dt} M(t) = R^{-1} \int \chi(\frac{r}{\delta T}) \psi(\frac{r}{R}) u_{t} u_{tr} r^{d} dr + \frac{d - 1}{2} R^{-1} \int \chi(\frac{r}{\delta T}) \psi(\frac{r}{R}) u_{t}^{2} r^{d - 1} dr \\
+ R^{-1} \int \chi(\frac{r}{\delta T}) \psi(\frac{r}{R}) u_{tt} u_{r} r^{d} dr + \frac{d - 1}{2} R^{-1} \int \chi(\frac{r}{\delta T}) \psi(\frac{r}{R}) u_{tt} u r^{d - 1} dr.
\endaligned
\end{equation}
Integrating by parts in $r$, by $(\ref{5.27})$,
\begin{equation}\label{5.33}
\aligned
R^{-1} \int \chi(\frac{r}{\delta T}) \psi(\frac{r}{R}) u_{t} u_{tr} r^{d} dr + \frac{d - 1}{2} R^{-1} \int \chi(\frac{r}{\delta T}) \psi(\frac{r}{R}) u_{t}^{2} r^{d - 1} dr \\ = -\frac{1}{2} \int \chi(\frac{r}{\delta T}) \phi(\frac{r}{R}) u_{t}^{2} r^{d - 1} dr + \frac{1}{\delta T} \int_{\delta T \leq |x| \leq 2 \delta T} u_{t}^{2} dx.
\endaligned
\end{equation}

Next, integrating by parts and using $(\ref{5.27})$,
\begin{equation}\label{5.34}
\aligned
R^{-1} \int \chi(\frac{r}{\delta T}) \psi(\frac{r}{R}) (u_{rr} + \frac{d - 1}{r} u_{r}) u_{r} r^{d} dr + \frac{d - 1}{2} R^{-1} \chi(\frac{r}{\delta T}) \psi(\frac{r}{R}) (u_{rr} + \frac{d - 1}{r} u_{r}) u r^{d - 1} dr \\
= -\frac{1}{2} \int \chi(\frac{r}{\delta T}) \phi(\frac{r}{R}) u_{r}^{2} r^{3} dr + \frac{1}{\delta T} \int_{\delta T \leq |x| \leq 2 \delta T} |\nabla u|^{2} dx + \frac{1}{R} \int \chi(\frac{r}{\delta T}) \psi(\frac{r}{R}) u^{2} r^{d - 3} dr + \frac{1}{\delta^{3} T^{3}} \int_{\delta T \leq |x| \leq 2 \delta T} u^{2} dx.
\endaligned
\end{equation}
Now then,
\begin{equation}\label{5.35}
\aligned
\frac{1}{\delta T} \int_{0}^{T} \int_{\delta T \leq |x| \leq 2 \delta T} |\nabla u|^{2} dx dt + \frac{1}{\delta^{3} T^{3}} \int_{0}^{T} \int_{\delta T \leq |x| \leq 2 \delta T} u^{2} dx dt \\ \lesssim_{\delta} \sup_{t \in [0, T]} \| \nabla_{t, x} u \|_{L^{2}(|x| \leq 2 \delta T)}^{2} + \sup_{t \in [0, T]} \frac{1}{\delta^{2} T^{2}} \| u \|_{L^{2}(|x| \leq 2 \delta T)}^{2}.
\endaligned
\end{equation}
Also,
\begin{equation}\label{5.36}
\frac{1}{R} \int_{0}^{T} \int \chi(\frac{r}{\delta T}) \psi(\frac{r}{R}) u^{2} r^{d - 3} dr dt \lesssim \int_{0}^{T} \int \chi(\frac{r}{\delta T}) \frac{1}{|x|^{3}} u^{2} dx dt,
\end{equation}
and we can use Proposition $\ref{p5.3}$ to estimate this term.\medskip

Next, integrating by parts,
\begin{equation}\label{5.37}
\aligned
-R^{-1} \int \chi(\frac{r}{\delta T}) \psi(\frac{r}{R})  |u|^{\frac{4}{d - 1}} u u_{r} r^{d} dr - \frac{d - 1}{2} R^{-1} \int \chi(\frac{r}{\delta T}) \psi(\frac{r}{R}) |u|^{\frac{2(d + 1)}{d - 1}} r^{d - 1} dr \\
=  -\frac{d - 1}{2(d + 1)R} \int \chi(\frac{r}{\delta T}) \psi(\frac{r}{R}) |u|^{\frac{2(d + 1)}{d - 1}} r^{3} dr + \frac{1}{R^{2}} \int \chi(\frac{r}{\delta T}) \psi'(\frac{r}{R}) |u|^{\frac{2(d + 1)}{d - 1}} r^{4} + \frac{1}{\delta T} \int_{\delta T \leq |x| \leq 2 \delta T} |u|^{\frac{2(d + 1)}{d - 1}} dx \\
\leq  -\frac{d - 1}{2R(d + 1)} \int \chi(\frac{r}{\delta T}) \psi(\frac{r}{R}) |u|^{\frac{2(d + 1)}{d - 1}} r^{d - 1} dr  + \frac{1}{\delta T} \int_{\delta T \leq |x| \leq 2 \delta T} |u|^{\frac{2(d + 1)}{d - 1}} dx.
\endaligned
\end{equation}
The last inequality uses the fact that $\psi'(r) \leq 0$ for all $r$. This completes the proof of the theorem.
\end{proof}

\section{Scattering in the $d = 4$ case}
The proofs of global well-posedness and scattering are slightly different in the $d = 4$ and $d > 4$ cases. We start with the $d = 4$ case.
\begin{theorem}\label{t2.1}
If $u$ is a solution to the conformal wave equation,
\begin{equation}\label{2.2}
u_{tt} - \Delta u + |u|^{\frac{4}{3}} u = 0, \qquad u(0,x) = u_{0} \in \dot{H}^{1/2}, \qquad u_{t}(0, x) = u_{1} \in \dot{H}^{-1/2}, \qquad u_{0}, u_{1} \qquad \text{radial},
\end{equation}
$u : \mathbb{R} \times \mathbb{R}^{4} \rightarrow \mathbb{R}$, then $u$ is a global solution to $(\ref{2.2})$ and scatters, that is
\begin{equation}\label{2.3}
\| u \|_{L_{t,x}^{\frac{10}{3}}(\mathbb{R} \times \mathbb{R}^{4})} \leq C(u_{0}, u_{1}) < \infty.
\end{equation}
\end{theorem}
\begin{remark}
Note that Theorem $\ref{t2.1}$ does not state that the bound on $(\ref{2.3})$ depends on the $\dot{H}^{1/2} \times \dot{H}^{-1/2}$ norm of the initial data, but rather depends on the actual initial data $(u_{0}, u_{1})$. The proof that a bound exists that is a function of the $\dot{H}^{1/2} \times \dot{H}^{-1/2}$ norm will utilize the profile decomposition.
\end{remark}
\begin{proof}
In this case, it will be helpful to begin with a more detailed explanation of the approximation analysis that will be used in every subsequent proof.\medskip

By time reversal symmetry, it is enough to show that $\| u \|_{L_{t,x}^{10/3}([0, \infty) \times \mathbb{R}^{4})} < \infty$. Furthermore, we can translate in time so that the initial data is at $t = 1$ and show that $\| u \|_{L_{t,x}^{10/3}([1, \infty) \times \mathbb{R}^{4})} < \infty$. This means that we do not need to worry about $t < 1$.\medskip

Again let $\chi \in C_{0}^{\infty}(\mathbb{R}^{4})$ be a smooth, cutoff function, $\chi(x) = 1$ for $|x| \leq 1$ and $\chi(x) = 0$ for $|x| \geq 2$. Then, split the initial data,
\begin{equation}\label{2.4}
u_{0} = v_{0} + w_{0}, \qquad u_{1} = v_{1} + w_{1},
\end{equation}
where
\begin{equation}\label{2.5}
v_{0} = \chi(\frac{x}{R}) P_{\leq N} u_{0}, \qquad v_{1} = \chi(\frac{x}{R}) P_{\leq N} u_{1},
\end{equation}
for some $0 < R < \infty$ and $0 < N < \infty$. Here $P_{\leq N}$ is the standard Littlewood--Paley projection to frequencies $\leq N$. By the dominated convergence theorem, there exists some $N < \infty$ such that
\begin{equation}\label{2.6}
\| P_{> N} u_{0} \|_{\dot{H}^{1/2}} + \| P_{> N} u_{1} \|_{\dot{H}^{-1/2}} \leq \frac{\epsilon}{2}.
\end{equation}
It is convenient to rescale the initial data so that $N = 1$.\medskip

After rescaling, by the dominated convergence theorem implies that there exists some $R < \infty$ such that
\begin{equation}\label{2.7}
\| \chi(\frac{x}{R}) P_{\leq 1} u_{0} \|_{\dot{H}^{1/2}} + \| \chi(\frac{x}{R}) P_{\leq 1} u_{1} \|_{\dot{H}^{-1/2}} \leq \frac{\epsilon}{2}.
\end{equation}
Therefore, if $\mathcal E(t)$ is the conformal energy of $v$,
\begin{equation}\label{2.8}
\aligned
\mathcal E(t) = \frac{1}{4} \int_{\mathbb{R}^{d}} |(t + |x|) Lv + (d - 1) v|^{2} + |(t - |x|) \underline{L} v + (d - 1) v|^{2} dx \\ + \frac{1}{2} \int (t^{2} + |x|^{2}) |\cancel{\nabla} v|^{2} dx + \frac{d - 1}{4(d + 1)} \int (t^{2} + |x|^{2}) |v|^{\frac{2(d + 1)}{d - 1}} dx,
\endaligned
\end{equation}
where $L = (\partial_{t} + \frac{x}{|x|} \cdot \nabla)$ and $\underline{L} = (\partial_{t} - \frac{x}{|x|} \cdot \nabla)$. Then by direct computation using the Fourier and spatial support of $v_{0}$ and $v_{1}$,
\begin{equation}\label{2.9}
\mathcal E(t)|_{t = 1} \lesssim R^{2} \| v_{0} \|_{\dot{H}^{1}}^{2} + R^{2} \| v_{1} \|_{L^{2}}^{2} + R^{2} \| v_{0} \|_{L^{10/3}}^{10/3} + \| v_{0} \|_{L^{2}}^{2} \lesssim R^{2}(\| v_{0} \|_{\dot{H}^{1/2}}^{2} + \| v_{1} \|_{\dot{H}^{-1/2}}^{2}).
\end{equation}

Now, for any $\sigma > 0$, define
\begin{equation}\label{2.10}
w_{0}^{\sigma} = \chi(\frac{\sigma x}{R}) P_{\leq \frac{1}{\sigma}} w_{0}, \qquad w_{1}^{\sigma} = \chi(\frac{\sigma x}{R}) P_{\leq \frac{1}{\sigma}} w_{1}.
\end{equation}
For any $\sigma > 0$,
\begin{equation}\label{2.11}
\mathcal E(v_{0} + w_{0}^{\sigma}, v_{1} + w_{1}^{\sigma}) < \infty.
\end{equation}
Therefore, by Theorem $\ref{t5.1}$, if $u^{\sigma}$ solves $(\ref{2.2})$ with initial data
\begin{equation}\label{2.12}
u_{0}^{\sigma} = v_{0} + w_{0}^{\sigma}, \qquad u_{1}^{\sigma} = v_{1} + w_{1}^{\sigma},
\end{equation}
then
\begin{equation}\label{2.13}
\| u^{\sigma} \|_{L_{t,x}^{10/3}} \leq C(\sigma, u_{0}, u_{1}) < \infty,
\end{equation}
and therefore $u^{\sigma}$ is globally well-posed and scattering.\medskip

To prove Theorem $\ref{t2.1}$, it suffices to prove $(\ref{2.13})$ holds with a bound that does not depend on $\sigma$, that is,
\begin{equation}\label{2.1}
\| u^{\sigma} \|_{L_{t,x}^{10/3}} \leq C(u_{0}, u_{1}) < \infty.
\end{equation}
 Then, since
\begin{equation}\label{2.14}
\lim_{\sigma \searrow 0} u_{0}^{\sigma} = u_{0}, \qquad \text{in} \qquad \dot{H}^{1/2}, \qquad \text{and} \qquad \lim_{\sigma \searrow 0} u_{1}^{\sigma} = u_{1}, \qquad \text{in} \qquad \dot{H}^{-1/2},
\end{equation}
in that case $(\ref{2.3})$ follows directly from standard perturbation theory. The reason for making the approximation of the initial data in $(\ref{2.12})$ is to guarantee that the solution $u^{\sigma}$ is global and scattering, so that we can make a bootstrap argument. We suppress the $\sigma$'s for the rest of the argument.
\begin{remark}
Please note that it is perfectly fine to prove a bound on $\| u \|_{L_{t,x}^{10/3}}$ that does depend on $R$, since $R$ is fixed as $\sigma \searrow 0$.
\end{remark}

Now decompose $(\ref{2.2})$ into a system of equations,
\begin{equation}\label{2.15}
\aligned
w_{tt} - \Delta w + |w|^{\frac{4}{3}} w &= 0, \qquad w(0, x) = w_{0}, \qquad w_{t}(0, x) = w_{1}, \\
v_{tt} - \Delta v + F &= 0, \qquad F = |u|^{\frac{4}{3}} u - |w|^{\frac{4}{3}} w, \qquad v(0, x) = v_{0}, \qquad v_{t}(0, x) = v_{1}.
\endaligned
\end{equation}
By $(\ref{2.6})$, $(\ref{4.5})$, Theorems $\ref{t4.2}$, $\ref{t4.3}$ and Lemmas $\ref{l4.4}$ and $\ref{l4.5}$,
\begin{equation}\label{2.16}
\| w \|_{L_{t}^{2} L_{x}^{4}(\mathbb{R} \times \mathbb{R}^{4})} + \| w \|_{L_{t,x}^{10/3}(\mathbb{R} \times \mathbb{R}^{4})} + \| |x|^{3/5} w \|_{L_{t}^{2} L_{x}^{10}(\mathbb{R} \times \mathbb{R}^{4})} + \| |x|^{11/10} w \|_{L_{t}^{\infty} L_{x}^{10}(\mathbb{R} \times \mathbb{R}^{4})} + \| |x| w \|_{L_{t,x}^{10}(\mathbb{R} \times \mathbb{R}^{4})} \lesssim \epsilon.
\end{equation}

By direct computation,
\begin{equation}\label{2.16}
\frac{d}{dt} \mathcal E(v) = -\langle (t + |x|) Lv + 3v, (t + |x|) (|u|^{\frac{4}{3}} u - |w|^{\frac{4}{3}} w - |v|^{\frac{4}{3}} v) \rangle - \langle (t - |x|) \underline{L}v + 3v, (t - |x|) (|u|^{\frac{4}{3}} u - |w|^{\frac{4}{3}} w - |v|^{\frac{4}{3}} v) \rangle.
\end{equation}
Also by direct computation,
\begin{equation}\label{2.17}
||u|^{\frac{4}{3}} u - |w|^{\frac{4}{3}} w - |v|^{\frac{4}{3}} v| \lesssim |w| |v| (|v|^{1/3} + |w|^{1/3}).
\end{equation}
Therefore,
\begin{equation}\label{2.19}
\| |x| (|u|^{\frac{4}{3}} u - |v|^{\frac{4}{3}} v - |w|^{\frac{4}{3}} w) \|_{L_{x}^{2}} \lesssim \| |x| w \|_{L_{x}^{10}} \| v \|_{L_{x}^{10/3}}^{4/3} + \| |x| w \|_{L_{x}^{10}} \| w \|_{L_{x}^{10/3}}^{1/3} \| v \|_{L_{x}^{10/3}},
\end{equation}
and
\begin{equation}\label{2.20}
\| t |w| \|_{L_{x}^{10}(|x| > \delta |t|)} \lesssim \frac{1}{\delta} \| |x| w \|_{L_{x}^{10}},
\end{equation}
so
\begin{equation}\label{2.21}
\| |t| ((|u|^{\frac{4}{3}} u - |v|^{\frac{4}{3}} v - |w|^{\frac{4}{3}} w) \|_{L_{x}^{2}(|x| > \delta |t|)} \lesssim \frac{1}{\delta} \| |x| w \|_{L_{x}^{10}} \| v \|_{L_{x}^{10/3}}^{4/3} + \frac{1}{\delta} \| |x| w \|_{L_{x}^{10}}^{1/3} \| w \|_{L_{x}^{10/3}}^{1/3} \| v \|_{L_{x}^{10/3}}.
\end{equation}
Since $\| v \|_{L_{x}^{10/3}}^{10/3} \lesssim \frac{\mathcal E(t)}{t^{2}}$,
\begin{equation}\label{2.22}
\aligned
\frac{d}{dt} \mathcal E(t) = -t \int_{|x| \leq \delta |t|} [ (t + |x|) Lv + 3v ] (|u|^{\frac{4}{3}} u - |w|^{\frac{4}{3}} w - |v|^{\frac{4}{3}} v) dx \\ - t \int_{|x| \leq \delta |t|} [(t - |x|) \underline{L}v + 3v] (|u|^{\frac{4}{3}} u - |w|^{\frac{4}{3}} w - |v|^{\frac{4}{3}} v) dx \\ + \frac{1}{\delta} \frac{\mathcal E(v)^{9/10}}{t^{4/5}} \| |x| w \|_{L_{x}^{10}} + \frac{1}{\delta} \frac{\mathcal E(v)^{4/5}}{t^{3/5}} \| |x| w \|_{L_{x}^{10}} \| w \|_{L_{x}^{10/3}}^{1/3}.
\endaligned
\end{equation}

Now then, by the fundamental theorem of calculus, for $t > \frac{1}{\delta^{1/2}}$,
\begin{equation}\label{2.23}
\mathcal E(t) = \mathcal E(\delta^{1/2} t) + \int_{\delta^{1/2} t}^{t} \frac{d}{d \tau} \mathcal E(\tau) d\tau,
\end{equation}
and for $1 < t < \frac{1}{\delta^{1/2}}$,
\begin{equation}\label{2.23.1}
\mathcal E(t) = \mathcal E(1) + \int_{1}^{t} \frac{d}{d \tau} \mathcal E(\tau) d\tau.
\end{equation}
Now, by a change of variables,
\begin{equation}\label{2.24}
\int_{\delta^{-1/2}}^{\infty} \frac{\mathcal E(\delta^{1/2} t)}{t^{2}} dt = \delta^{1/2} \int_{\delta^{-1/2}}^{\infty} \frac{\delta^{1/2}}{\delta t^{2}} \mathcal E(\delta t) \delta^{1/2} dt = \delta^{1/2} \int_{1}^{\infty} \frac{\mathcal E(t)}{t^{2}} dt,
\end{equation}
and
\begin{equation}\label{2.24.1}
\int_{1}^{\delta^{-1/2}} \frac{\mathcal E(0)}{t^{2}} dt \leq \mathcal E(0).
\end{equation}
Next, by Fubini's theorem, letting $t' = \sup \{ 1, \delta^{1/2} t \}$ to simplify notation,
\begin{equation}\label{2.25}
\aligned
\int_{1}^{\infty} \frac{1}{t^{2}} \int_{t'}^{t} \frac{1}{\delta} \frac{\mathcal E(s)^{9/10}}{s^{4/5}} \| |x| w(s) \|_{L_{x}^{10}} ds dt + \int_{1}^{\infty} \frac{1}{t^{2}} \int_{t'}^{t} \frac{1}{\delta} \frac{\mathcal E(s)^{4/5}}{s^{3/5}} \| |x| w(s) \|_{L_{x}^{10}} \| w \|_{L_{x}^{10/3}}^{1/3} ds dt \\
\lesssim \int_{1}^{\infty} \frac{1}{\delta} \frac{\mathcal E(t)^{9/10}}{t^{9/5}} \| |x| w(t) \|_{L_{x}^{10}} dt + \int_{0}^{\infty} \frac{1}{\delta} \frac{\mathcal E(t)^{4/5}}{t^{8/5}} \| |x| w(t) \|_{L_{x}^{10}} \| w(t) \|_{L_{x}^{10/3}}^{1/3} dt \\
\lesssim \frac{1}{\delta} (\int_{0}^{\infty} \frac{\mathcal E(t)}{t^{2}} dt)^{9/10} \| |x| w \|_{L_{t,x}^{10}} + \frac{1}{\delta} (\int_{0}^{\infty} \frac{\mathcal E(t)}{t^{2}} dt)^{4/5} \| |x| w \|_{L_{t,x}^{10}} \| w \|_{L_{t,x}^{10/3}}^{1/3} \\ \lesssim \frac{\epsilon}{\delta} (\int_{0}^{\infty} \frac{\mathcal E(t)}{t^{2}} dt)^{9/10} + \frac{\epsilon^{4/3}}{\delta} (\int_{0}^{\infty} \frac{\mathcal E(t)}{t^{2}} dt)^{4/5}.
\endaligned
\end{equation}

The terms with $(t + |x|) Lv$ and $(t - |x|) \underline{L} v$ may be handled in exactly the same way, so using $(\ref{2.17})$ and $(\ref{2.22})$, it remains to compute
\begin{equation}\label{2.26}
 \int_{t'}^{t} \int_{|x| \leq \delta |t|} \tau [(\tau + |x|) Lv + 3v] |w| |v|^{4/3} dx d\tau, \qquad \text{and} \qquad \int_{t'}^{t} \int_{|x| \leq \delta |t|} \tau [(\tau + |x|) Lv + 3v] |w|^{4/3} |v| dx d\tau,
\end{equation}
separately.\medskip

It is convenient to replace $\tilde{\mathcal E}(t)$ by
\begin{equation}\label{2.27}
\tilde{\mathcal E}(t) = \sup_{0 \leq s \leq t} \mathcal E(t).
\end{equation}
Note that $\frac{d}{dt} \tilde{\mathcal E}(t)$ is bounded by the right hand side of the absolute value of $(\ref{2.22})$. We abuse notation and let $\mathcal E(t)$ denote $\tilde{\mathcal E}(t)$.\medskip 

Now, let $\mathcal R = \inf \{ (\frac{\mathcal E(t)}{t^{2}} + \frac{t^{2/5} \mathcal E(t)^{3/5}}{t^{2}})^{-1} , \delta |t| \}$. By H{\"o}lder's inequality,
\begin{equation}\label{2.28}
\aligned
t \int_{t'}^{t} \int_{|x| \leq \mathcal R} [(\tau + |x|) Lv + 3v] |w| |v|^{4/3} dx d\tau \\ \lesssim t \| [(\tau + |x|) Lv + 3v \|_{L_{\tau}^{\infty} L_{x}^{2}(|x| \leq \mathcal R)}^{1/2} \| |x|^{-1/6} \{(\tau + |x|) Lv + 3v \} \|_{L_{\tau, x}^{2}(|x| \leq \mathcal R)}^{1/2} \\ \times \| w \|_{L_{\tau}^{2} L_{x}^{4}} \| \frac{1}{|x|^{3/2}} v \|_{L_{\tau, x}^{2}(|x| \leq \mathcal R)}^{1/2} \| |x| v \|_{L_{t,x}^{\infty}(|x| \leq \mathcal R)}^{5/6} \\
\lesssim t \mathcal E(t)^{1/4} \| |x|^{-1/6} \{(\tau + |x|) Lv + 3v \} \|_{L_{\tau, x}^{2}(|x| \leq \mathcal R)}^{1/2} \| w \|_{L_{\tau}^{2} L_{x}^{4}} \| \frac{1}{|x|^{3/2}} v \|_{L_{\tau, x}^{2}(|x| \leq \mathcal R)}^{1/2} \| |x| v \|_{L_{t,x}^{\infty}(|x| \leq \mathcal R)}^{5/6},
\endaligned
\end{equation}
\begin{remark}
All time intervals are $[t', t]$ where $t' = \sup \{ 1, \delta^{1/2} t \}$.
\end{remark}
and
\begin{equation}\label{2.29}
\aligned
t \int_{\delta^{1/2} t}^{t} \int_{\mathcal R \leq |x| \leq \delta |t|} [(\tau + |x|) Lv + 3v] |w| |v|^{4/3} dx d\tau \\ 
\lesssim t \| [(\tau + |x|) Lv + 3v \|_{L_{\tau}^{\infty} L_{x}^{2}(|x| \leq \delta |t|)}^{4/5} \| |x|^{-1} (\tau + |x|) Lv + 3v  \|_{L_{\tau, x}^{2}(|x| \leq \delta |t|)}^{1/5} \| |x|^{3/5} w \|_{L_{\tau}^{2} L_{x}^{10}} \| \frac{1}{|x|^{3/10}} v \|_{L_{\tau, x}^{10/3}(|x| \leq \delta |t|)}^{4/3} \\
\lesssim  t \mathcal E(t)^{2/5} \| |x|^{-1} \{(\tau + |x|) Lv + 3v \}  \|_{L_{\tau, x}^{2}(\mathcal R \leq |x| \leq \delta |t|)}^{1/5} \| |x|^{3/5} w \|_{L_{\tau}^{2} L_{x}^{10}} \| \frac{1}{|x|^{3/10}} v \|_{L_{\tau, x}^{10/3}(|x| \leq \delta |t|)}^{4/3}.
\endaligned
\end{equation}

Next,
\begin{equation}\label{2.30}
\aligned
t \int_{\delta^{1/2} t}^{t} \int_{|x| \leq \delta |t|} [(\tau + |x|) Lv + 3v] |w|^{4/3} |v| dx d\tau \\ \lesssim t \| |x|^{3/5} w \|_{L_{t}^{2} L_{x}^{10}}^{4/3} \| |x|^{-3/2} v \|_{L_{t,x}^{2}}^{2/3} \| (\tau + |x|) Lv + 3v \|_{L_{\tau}^{\infty} L_{x}^{2}} \| |x|^{3/5} v \|_{L_{\tau}^{\infty} L_{x}^{10}}^{1/3} \\
\lesssim t \mathcal E(t)^{1/2} \| |x|^{3/5} w \|_{L_{t}^{2} L_{x}^{10}}^{4/3} \| |x|^{-3/2} v \|_{L_{t,x}^{2}}^{2/3} \| |x| v \|_{L_{\tau, x}^{\infty}(|x| \leq \delta |t|)}^{1/5} \| v \|_{L_{\tau}^{\infty} L_{x}^{4}(|x| \leq \delta |t|)}^{2/15}.
\endaligned
\end{equation}
By the Sobolev embedding theorem and radial Sobolev embedding theorem,
\begin{equation}\label{2.31}
\| |x| v \|_{L_{x}^{\infty}(|x| \leq \delta |t|)}, \| v \|_{L_{x}^{4}(|x| \leq \delta |t|)} \lesssim \| \chi(\frac{x}{\delta |t|}) v \|_{\dot{H}^{1}} \lesssim \| \chi(\frac{x}{\delta |t|}) \nabla v \|_{L^{2}} + \frac{1}{\delta |t|} \| v \|_{L_{x}^{2}(|x| \leq 2 \delta |t|)}.
\end{equation}
Now, for $\delta^{1/2} t \leq \tau \leq t$,
\begin{equation}\label{2.32}
\tau \| \chi(\frac{x}{\delta t}) \nabla_{t,x} v(\tau) \|_{L^{2}} \lesssim \| \chi(\frac{x}{\delta t}) \{ (\tau + |x|) Lv + 3v \} \|_{L^{2}} + \| \chi(\frac{x}{\delta t}) \{ (\tau - |x|) \underline{L} v + 3v \} \|_{L^{2}} + \| \chi(\frac{x}{\delta t}) v \|_{L^{2}}.
\end{equation}
Now, by H{\"o}lder's inequality and the conformal energy,
\begin{equation}\label{2.33}
\| \chi(\frac{x}{\delta t}) v(\tau) \|_{L^{2}} \lesssim (\delta t)^{4/5} \| v \|_{L_{x}^{10/3}} \lesssim (\delta t)^{4/5} \frac{\mathcal E(\tau)^{3/10}} {\tau^{3/5}} \lesssim \delta^{1/2} t^{1/5} \mathcal E(\tau)^{3/10}.
\end{equation}
Plugging $(\ref{2.33})$ into $(\ref{2.32})$,
\begin{equation}\label{2.34}
\tau \| \chi(\frac{x}{\delta t}) \nabla_{t,x} v(\tau) \|_{L^{2}} \lesssim \mathcal E(\tau)^{1/2} + \delta^{1/2} t^{1/5} \mathcal E(\tau)^{3/10}.
\end{equation}
Plugging $(\ref{2.34})$ into $(\ref{2.31})$,
\begin{equation}\label{2.35}
\| |x| v(\tau) \|_{L_{x}^{\infty}(|x| \leq \delta |t|)}, \| v(\tau) \|_{L_{x}^{4}(|x| \leq \delta |t|)} \lesssim \frac{1}{\tau} \mathcal E(\tau)^{1/2} + \frac{\delta^{1/2} t^{1/5} \mathcal E(\tau)^{3/10}}{\tau} + \frac{\delta^{1/2} t^{1/5} \mathcal E(\tau)^{3/10}}{\delta t}.
\end{equation}
Again using the fact that $\delta^{1/2} t \leq \tau \leq t$ and $\mathcal E(t)$ is increasing,
\begin{equation}\label{2.36}
\| |x| v(\tau) \|_{L_{x}^{\infty}(|x| \leq \delta t)}, \| v(\tau) \|_{L_{x}^{4}(|x| \leq \delta t)} \lesssim_{\delta} \frac{\mathcal E(t)^{1/2}}{t} + \frac{t^{1/5} \mathcal E(t)^{3/10}}{t}.
\end{equation}

Next, we utilize the Morawetz estimate in Proposition $\ref{p5.3}$ and the local energy estimate in Proposition $\ref{p5.4}$.
\begin{proposition}\label{p2.2}
For any $T > 1$, $T' = \sup \{ 1, \delta^{1/2} T \}$,
\begin{equation}\label{2.37}
\aligned
\int_{T'}^{T} \int \chi(\frac{x}{\delta T}) [\frac{1}{|x|^{3}} v^{2} + \frac{1}{|x|} |v|^{10/3}] dx dt \lesssim_{\delta} \frac{\mathcal E(T)}{T^{2}} + \frac{T^{2/5} \mathcal E(T)^{3/5}}{T^{2}} \\ + \int_{T'}^{T} \int \chi(\frac{x}{\delta T}) |\nabla v| |v| |w| (|v|^{1/3} + |w|^{1/3}) dx dt + \int_{T'}^{T} \int \chi(\frac{x}{\delta T}) |v|^{2} |w| (|v|^{1/3} + |w|^{1/3}) \frac{1}{|x|} dx dt.
\endaligned
\end{equation}
\end{proposition}
\begin{proof}
Recalling $(\ref{2.5})$ and the proof of Proposition $\ref{p5.3}$,
\begin{equation}\label{2.38}
v_{tt} - \Delta v + |v|^{\frac{4}{3}} v + [F - |v|^{\frac{4}{3}} v] = 0,
\end{equation}
\begin{equation}\label{2.39}
\aligned
\int_{T'}^{T} \int \chi(\frac{x}{\delta T}) [\frac{1}{|x|^{3}} v^{2} + \frac{1}{|x|} |v|^{10/3}] dx dt \lesssim_{\delta} \sup_{t \in [0, T]} \| \nabla_{t, x} u \|_{L^{2}(|x| \leq 2 \delta T)}^{2} + \sup_{t \in [0, T]} \frac{1}{\delta^{2} T^{2}} \| u \|_{L^{2}(|x| \leq 2 \delta T)}^{2} \\ + \int_{T'}^{T} \int \chi(\frac{x}{\delta T}) |\nabla v| [F - |v|^{\frac{4}{3}} v] dx dt + \int_{T'}^{T} \int \chi(\frac{x}{\delta T}) |v| [F - |v|^{\frac{4}{3}} v] \frac{1}{|x|} dx dt.
\endaligned
\end{equation}
Now then, by $(\ref{2.31})$--$(\ref{2.34})$,
\begin{equation}\label{2.40}
\sup_{t \in [0, T]} \| \nabla_{t, x} u \|_{L^{2}(|x| \leq 2 \delta T)}^{2} + \sup_{t \in [0, T]} \frac{1}{\delta^{2} T^{2}} \| u \|_{L^{2}(|x| \leq 2 \delta T)}^{2} \lesssim_{\delta} \frac{\mathcal E(T)}{T^{2}} + \frac{T^{2/5} \mathcal E(T)^{3/5}}{T^{2}}.
\end{equation}
Finally, plugging in the bounds in $(\ref{2.17})$ to the final two terms in the right hand side of $(\ref{2.39})$ proves the theorem.
\end{proof}

Next we prove a local energy decay estimate.
\begin{proposition}[Local energy decay]\label{p2.3}
For any $T > 1$, $R > 0$, if $T' = \sup \{ 1, \delta^{1/2} T \}$,
\begin{equation}\label{2.50}
\aligned
R^{-1} \int_{T'}^{T} \int_{|x| \leq R} \chi(\frac{x}{\delta T}) [|\nabla v|^{2} + v_{t}^{2}] dx dt \lesssim_{\delta} \frac{\mathcal E(T)}{T^{2}} + \frac{T^{2/5} \mathcal E(T)^{3/5}}{T^{2}} \\ + \int_{T'}^{T} \int \chi(\frac{x}{\delta T}) \psi(\frac{r}{R}) \frac{r}{R} |\nabla v| |v| |w| (|v|^{1/3} + |w|^{1/3}) dx dt \\ + \frac{1}{R} \int_{T'}^{T} \int \chi(\frac{x}{\delta T}) \psi(\frac{r}{R}) |v|^{2} |w| (|v|^{1/3} + |w|^{1/3}) dx dt.
\endaligned
\end{equation}
\end{proposition}
\begin{proof}
As in Proposition $\ref{p2.2}$, the proof follows directly from the proof of Proposition $\ref{p5.4}$, $(\ref{2.17})$, $(\ref{2.38})$, and $(\ref{2.40})$.
\end{proof}

Next, we show that we can actually ignore absorb the error terms in Propositions $\ref{p2.2}$ and $\ref{p2.3}$ into the left hand side.
\begin{proposition}\label{p2.4}
For $T > 1$, $T' = \sup \{ 1, \delta^{1/2} T \}$,
\begin{equation}\label{2.65}
\aligned
\sup_{R > 0} R^{-1} \int_{T'}^{T} \int_{|x| \leq R} \chi(\frac{x}{\delta T}) [|\nabla v|^{2} + v_{t}^{2}] dx dt + \int_{T'}^{T} \int \chi(\frac{x}{\delta T}) [\frac{1}{|x|^{3}} v^{2} + \frac{1}{|x|} |v|^{10/3}] dx dt \lesssim_{\delta} \frac{\mathcal E(T)}{T^{2}} + \frac{T^{2/5} \mathcal E(T)^{3/5}}{T^{2}}.
\endaligned
\end{equation}
\end{proposition}
\begin{proof}
By Propositions $\ref{p2.2}$ and $\ref{p2.3}$, it only remains to estimate the error terms. First, for any $\eta > 0$, by $(\ref{2.16})$,
\begin{equation}\label{2.66}
\aligned
\int_{T'}^{T} \int \chi(\frac{x}{\delta T}) |v|^{2} |w| (|v|^{1/3} + |w|^{1/3}) \frac{1}{|x|} dx dt \\ \lesssim (\int \int \chi(\frac{x}{\delta T}) \frac{1}{|x|} |v|^{10/3})^{1/4} (\int \int \chi(\frac{x}{\delta T}) \frac{1}{|x|^{3}} |v|^{2})^{1/4} \| w \|_{L_{t}^{2} L_{x}^{4}} \| \chi(\frac{x}{\delta T})^{1/2} v \|_{L_{t}^{\infty} L_{x}^{4}} \\ \lesssim \eta (\int_{T'}^{T} \int \chi(\frac{x}{\delta T}) |v|^{10/3} dx dt) + \eta(\int_{T'}^{T} \int \chi(\frac{x}{\delta T}) \frac{1}{|x|^{3}} v^{2} dx dt) + C(\eta) \epsilon^{2} \| v \|_{L_{t}^{\infty} L_{x}^{4}(|x| \leq 2 \delta T)}^{2}.
\endaligned
\end{equation}
For $\eta(\delta) \ll 1$ the first two terms in $(\ref{2.66})$ can be absorbed into the left hand side of $(\ref{2.65})$. Now, by $(\ref{2.36})$,
\begin{equation}\label{2.67}
C(\eta) \| v \|_{L_{t}^{\infty} L_{x}^{4}(|x| \leq 2 \delta T)}^{2} \| w \|_{L_{t}^{2} L_{x}^{4}}^{2} \lesssim C(\eta) \epsilon^{2} (\frac{\mathcal E(T)}{T^{2}} + \frac{T^{2/5} \mathcal E(T)^{3/5}}{T^{2}}).
\end{equation}

Next, for any $R > 0$,
\begin{equation}\label{2.68}
\frac{1}{R} \psi(\frac{r}{R}) \lesssim \frac{1}{r},
\end{equation}
so as in $(\ref{2.66})$,
\begin{equation}\label{2.69}
\aligned
\sup_{R > 0} \frac{1}{R} \int_{T'}^{T} \int \chi(\frac{x}{\delta T}) \psi(\frac{r}{R}) |v|^{2} |w| (|v|^{1/3} + |w|^{1/3}) dx dt \\
\lesssim \eta (\int_{T'}^{T} \int \chi(\frac{x}{\delta T}) |v|^{10/3} dx dt) + \eta(\int_{T'}^{T} \int \frac{1}{|x|^{3}} v^{2} dx dt) + C(\eta) \| v \|_{L_{t}^{\infty} L_{x}^{4}(|x| \leq 2 \delta T)}^{2} \| w \|_{L_{t}^{2} L_{x}^{4}}^{2}.
\endaligned
\end{equation}

Next, using $(\ref{2.28})$--$(\ref{2.30})$ and splitting into the cases $|x| \leq \mathcal R$ and $|x| > \mathcal R$ separately,
\begin{equation}\label{2.70}
\aligned
\int_{T'}^{T} \int \chi(\frac{x}{\delta T}) |\nabla v| |v| |w|(|v|^{1/3} + |w|^{1/3}) dx dt \\
\lesssim \| \chi(\frac{x}{\delta T})^{1/2} \nabla v \|_{L_{t}^{\infty} L_{x}^{2}}^{1/2} \mathcal R^{1/6} (\sup_{R > 0} R^{-1/2} \| \chi(\frac{x}{\delta T})^{1/2} \nabla v \|_{L_{t,x}^{2}(|x| \leq R)})^{1/2}  \| w \|_{L_{t}^{2} L_{x}^{4}} \\ \times \| \chi(\frac{x}{\delta T})^{1/2} \frac{1}{|x|^{3/2}} v \|_{L_{t,x}^{2}}^{1/2} \| \chi(\frac{x}{\delta T})^{1/4} |x|^{5/6} v^{5/6} \|_{L_{t,x}^{\infty}} \\
+ \| \chi(\frac{x}{\delta T})^{1/2} \nabla v \|_{L_{t}^{\infty} L_{x}^{2}}^{4/5} \mathcal R^{-1/10} (\sup_{R > 0} R^{-1/2} \| \chi(\frac{x}{\delta T})^{1/2} \nabla v \|_{L_{t,x}^{2}(|x| \leq R)})^{1/5} \| |x|^{3/5} w \|_{L_{t}^{2} L_{x}^{10}} \| \chi(\frac{x}{\delta T})^{3/10} \frac{1}{|x|^{3/10}} v \|_{L_{t,x}^{10/3}}^{4/3} \\
+ \| \nabla v \|_{L_{t}^{\infty} L_{x}^{2}(|x| \leq 2 \delta T)} \| \chi(\frac{x}{\delta T})^{1/2} \frac{1}{|x|^{3/2}} v \|_{L_{t,x}^{2}}^{2/3} \| |x|^{3/5} w \|_{L_{t}^{2} L_{x}^{10}}^{4/3} \| |x|^{3/5} v \|_{L_{t}^{\infty} L_{x}^{10}}^{1/3}
\endaligned
\end{equation}
\begin{equation}\label{2.71}
\aligned
\lesssim \eta (\sup_{R > 0} R^{-1} \| \chi(\frac{x}{\delta T})^{1/2} \nabla v \|_{L_{t,x}^{2}}^{2}) + \eta(\int_{T'}^{T} \int \chi(\frac{x}{\delta T}) \frac{1}{|x|^{3}} |v|^{2} dx dt) \\ + \eta (\int_{T'}^{T} \int \chi(\frac{x}{\delta T}) \frac{1}{|x|} |v|^{10/3} dx dt) + C(\eta) \epsilon^{2} (\frac{\mathcal E(T)}{T^{2}} + \frac{T^{2/5} \mathcal E(T)^{3/5}}{T^{2}}).
\endaligned
\end{equation}
For $\eta(\delta) > 0$ sufficiently small, we can absorb the first three terms of $(\ref{2.71})$ into the left hand side of $(\ref{2.65})$, which implies that $(\ref{2.71})$ is controlled by the right hand side of $(\ref{2.65})$.\medskip

Finally, since $\psi(\frac{r}{R}) \frac{r}{R} \lesssim 1$, so
\begin{equation}\label{2.72}
\sup_{R > 0} \frac{1}{R} \int_{\delta^{1/2} T}^{T} \int \chi(\frac{x}{\delta T}) \psi(\frac{r}{R}) |v|^{2} |w| (|v|^{1/3} + |w|^{1/3}) r dx dt \lesssim (\ref{2.70}),
\end{equation}
which proves the theorem.
\end{proof}

The above computations may also be used to estimate $(\ref{2.28})$--$(\ref{2.30})$. First, by Proposition $\ref{p2.3}$,
\begin{equation}\label{2.47}
(\ref{2.30}) \lesssim_{\delta} \mathcal E(t) \| |x|^{3/5} w \|_{L_{t}^{2} L_{x}^{10}([\delta^{1/2} t, t] \times \mathbb{R}^{4})} + t^{1/5} \mathcal E(t)^{4/5} \| |x|^{3/5} w \|_{L_{t}^{2} L_{x}^{10}([\delta^{1/2} t, t] \times \mathbb{R}^{4})}.
\end{equation}
Multiplying this term by $\frac{1}{t^{2}}$ and integrating in time, choosing $\epsilon \ll \delta$, by $(\ref{2.16})$,
\begin{equation}\label{2.48}
\int_{1}^{\infty} C(\delta) \frac{\mathcal E(t)}{t^{2}} \| |x|^{3/5} w \|_{L_{t}^{2} L_{x}^{10}([\delta^{1/2} t, \delta t] \times \mathbb{R}^{4})} \lesssim C(\delta) \epsilon \int_{1}^{\infty} \frac{\mathcal E(t)}{t^{2}} dt \ll \int_{0}^{\infty} \frac{\mathcal E(t)}{t^{2}} dt.
\end{equation}
Also,
\begin{equation}\label{2.49}
\aligned
\int_{1}^{\infty} C(\delta) \frac{\mathcal E(t)^{4/5}}{t^{9/5}} \| |x|^{3/5} w \|_{L_{t}^{2} L_{x}^{10}([\delta^{1/2} t, \delta t] \times \mathbb{R}^{4})} dt \\ \lesssim C(\delta) (\int_{1}^{\infty} \frac{\mathcal E(t)}{t^{2}} dt)^{4/5} (\int_{1}^{\infty} \frac{1}{t} \| |x|^{3/5} w \|_{L_{t}^{2} L_{x}^{10}([\delta^{1/2} t, t] \times \mathbb{R}^{4})}^{5} dt)^{1/5} \\
\lesssim \epsilon C(\delta) (\int_{1}^{\infty} \frac{\mathcal E(t)}{t^{2}} dt)^{4/5} \ll \int_{1}^{\infty} \frac{\mathcal E(t)}{t^{2}} dt + 1.
\endaligned
\end{equation}

Turning now to $(\ref{2.28})$ and $(\ref{2.29})$,
\begin{equation}\label{2.61}
\aligned
t \mathcal E(t)^{1/4} \| |x|^{-1/6} \{(\tau + |x|) Lv + 3v \} \|_{L_{\tau, x}^{2}(|x| \leq \mathcal R)}^{1/2} \| w \|_{L_{\tau}^{2} L_{x}^{4}} \| \frac{1}{|x|^{3/2}} v \|_{L_{\tau, x}^{2}(|x| \leq \mathcal R)}^{1/2} \| |x| v \|_{L_{t,x}^{\infty}(|x| \leq \mathcal R)}^{5/6} \\
\lesssim_{\delta} t^{3/2} \mathcal E(t)^{1/4} \mathcal R^{1/6} (\frac{\mathcal E(t)}{t^{2}} + \frac{t^{2/5} \mathcal E(t)^{3/5}}{t^{2}})^{11/12} \| w \|_{L_{t}^{2} L_{x}^{4}}  \\ = t^{3/2} \mathcal E(t)^{1/4}  (\frac{\mathcal E(t)}{t^{2}} + \frac{t^{2/5} \mathcal E(t)^{3/5}}{t^{2}})^{3/4} \| w \|_{L_{t}^{2} L_{x}^{4}}
\lesssim (\mathcal E(t) + t^{2/5} \mathcal E(t)^{3/5}) \| w \|_{L_{t}^{2} L_{x}^{4}}.
\endaligned
\end{equation}
Now then,
\begin{equation}\label{2.62}
\aligned
C(\delta) \int_{0}^{\infty} \frac{\mathcal E(t)}{t^{2}} \| w \|_{L_{t}^{2} L_{x}^{4}} dt + C(\delta) \int_{0}^{\infty} \frac{\mathcal E(t)^{3/5}}{t^{8/5}} \| w \|_{L_{t}^{2} L_{x}^{4}([\delta^{1/2} t, t] \times \mathbb{R}^{4})} dt \\
\lesssim C(\delta) \epsilon \int_{0}^{\infty} \frac{\mathcal E(t)}{t^{2}} dt + C(\delta) (\int_{0}^{\infty} \frac{\mathcal E(t)}{t^{2}} dt)^{3/5} (\int_{0}^{\infty} \frac{1}{t} \| w \|_{L_{t}^{2} L_{x}^{4}([\delta^{1/2} t, t] \times \mathbb{R}^{4})}^{5/2} dt)^{2/5} \\
\ll \int_{0}^{\infty} \frac{\mathcal E(t)}{t^{2}} dt + 1.
\endaligned
\end{equation}

\begin{equation}\label{2.63}
\aligned
t \mathcal E(t)^{2/5} \| |x|^{-1} \{(\tau + |x|) Lv + 3v \}  \|_{L_{\tau, x}^{2}(\mathcal R \leq |x| \leq \delta |t|)}^{1/5} \| |x|^{3/5} w \|_{L_{\tau}^{2} L_{x}^{10}} \| \frac{1}{|x|^{3/10}} v \|_{L_{\tau, x}^{10/3}(|x| \leq \delta |t|)}^{4/3} \\
 \lesssim_{\delta} t^{6/5} \mathcal R^{-1/10} \mathcal E(t)^{2/5} (\frac{\mathcal E(t)}{t^{2}} + \frac{t^{2/5} \mathcal E(t)^{3/5}}{t^{2}})^{1/2} \| |x|^{3/5} w \|_{L_{t}^{2} L_{x}^{10}} \lesssim (\mathcal E(t) + t^{1/5} \mathcal E(t)^{4/5}) \| |x|^{3/5} w \|_{L_{t}^{2} L_{x}^{10}}.
 \endaligned
\end{equation}
Then,
\begin{equation}\label{2.64}
\aligned
C(\delta) \int_{0}^{\infty} \frac{\mathcal E(t)}{t^{2}} \| w \|_{L_{\tau}^{2} L_{x}^{4}([\delta^{1/2} t, t] \times \mathbb{R}^{4})} dt + C(\delta) \int_{0}^{\infty} \frac{t^{1/5} \mathcal E(t)^{4/5}}{t^{2}} \| w \|_{L_{\tau}^{2} L_{x}^{4}([\delta^{1/2} t, t] \times \mathbb{R}^{4})} dt \\
\lesssim C(\delta) \epsilon \int_{0}^{\infty} \frac{\mathcal E(t)}{t^{2}} dt + C(\delta) (\int_{0}^{\infty} \frac{\mathcal E(t)}{t^{2}} dt)^{4/5} (\int_{0}^{\infty} \frac{1}{t} \| |x|^{3/5} w \|_{L_{\tau}^{2} L_{x}^{10}([\delta^{1/2} t, t] \times \mathbb{R}^{4})}^{5} dt)^{1/5}
\ll \int_{0}^{\infty} \frac{\mathcal E(t)}{t^{2}} dt + 1.
\endaligned
\end{equation}

Therefore, combining $(\ref{2.61})$--$(\ref{2.64})$ with $(\ref{2.23})$--$(\ref{2.25})$ implies
\begin{equation}\label{2.65}
\int_{1}^{\infty} \frac{\mathcal E(t)}{t^{2}} dt \leq \mathcal E(1) + \eta \int_{1}^{\infty} \frac{\mathcal E(t)}{t^{2}} dt,
\end{equation}
for some $0 < \eta \ll 1$. Observe that $\mathcal E(1)$ depends on $R$, and $R(\epsilon) \gg 1$ for $\epsilon \ll 1$, but crucially the norm in $(\ref{2.65})$ does not depend on the $\sigma > 0$ in $(\ref{2.11})$.
\end{proof}

\section{A modified small data argument}
Extending the argument for the $d = 4$ case to dimensions $d > 4$ has a number of technical complications due to the low power of the nonlinearity $|u|^{\frac{4}{d - 1}} u$ when $d > 4$.\medskip

Indeed, for $d > 4$, if $u = v + w$,
\begin{equation}\label{7.1}
|u|^{\frac{4}{d - 1}} u - |v|^{\frac{4}{d - 1}} v - |w|^{\frac{4}{d - 1}} w \lesssim \inf \{ |w| |v|^{\frac{4}{d - 1}}, |v| |w|^{\frac{4}{d - 1}} \}.
\end{equation}
Therefore, when computing the first two lines of $(\ref{2.22})$ for $d > 5$, if we place $w \in L_{t}^{2} X$, where $X$ is some weighted $L^{p}$ space and $v \in L_{t}^{2} Y$, where $Y$ is some weighted $L^{q}$ space, we would need to place $(t - |x|) Lv \in L_{t}^{r} Z$ for some $r < \infty$ and $Z$ another weighted Banach space.\medskip

The obvious candidate for this would be to use the local energy decay estimate in $(\ref{2.50})$. However, $(\ref{2.50})$ does not quite give a bound on a weighted $L^{2}$ space. Instead, $(\ref{2.50})$ only implies
\begin{equation}\label{7.2}
\| |x|^{-1/2} \{ (t - |x|) Lv \} \|_{L_{t,x}^{2}(\frac{1}{T} \leq |x| \leq \delta T)}^{2} \lesssim \ln(T) (\frac{\mathcal E(T)}{T^{2}} + \frac{T^{2/5} \mathcal E(T)^{3/5}}{T^{2}}).
\end{equation}
\begin{remark}
In dimensions $d > 4$, $T^{2/5} \mathcal E(T)^{3/5}$ will be $T^{\alpha(d)} \mathcal E(T)^{1 - \alpha(d)}$ for some $\alpha(d) \searrow 0$ as $d \rightarrow \infty$, but this is not too important to the discussion right now.
\end{remark}

To work around the logarithmic divergence, we would like to use an argument similar to the argument in $(\ref{2.28})$ and $(\ref{2.29})$, namely to set
\begin{equation}\label{7.3}
\mathcal R = (\frac{\mathcal E(T)}{T^{2}} + \frac{T^{\alpha} \mathcal E(T)^{1 - \alpha}}{T^{2}})^{-1},
\end{equation}
and consider the cases $|x| \leq \mathcal R$ and $|x| > \mathcal R$ separately.\medskip

When $|x| > \mathcal R$, the computations are pretty similar to the $d = 4$ case. Taking $|w| |v|^{\frac{4}{d - 1}}$ in $(\ref{7.1})$ gives $(\frac{\mathcal E(T)}{T^{2}} + \frac{T^{\alpha} \mathcal E(T)^{1 - \alpha}}{T^{2}})^{\beta}$ for some $0 < \beta < 1$ along with $\mathcal R^{\beta - 1}$, and then we can proceed as in the $d = 4$ case. However, for $|x| \leq \mathcal R$, by $(\ref{7.1})$ we at most have a second order power of $v$. Therefore, in that case we cannot copy the analysis for $(\ref{2.28})$ and obtain $\mathcal (\frac{\mathcal E(T)}{T^{2}} + \frac{T^{\alpha} \mathcal E(T)^{1 - \alpha}}{T^{2}})^{1 + \beta'} R^{\beta'}$ for some $\beta' > 0$.
\begin{remark}
Observe for example that the computations in $(\ref{2.61})$ relied very heavily on the fact that the error term considered was of the form $|v|^{\frac{4}{3}} |w|$ and $\frac{4}{3} > 1$.
\end{remark}

What comes to the rescue is that, since $w$ is a solution to the small data problem, $v$ should usually be larger than $w$, and when it is not, we can put that part with the equation for $w$ at minimal cost. Instead split $u = v + w$, where
\begin{equation}\label{7.4}
\aligned
v_{tt} - \Delta v + (1 - \chi(\frac{u}{w})) |u|^{\frac{4}{d - 1}} u &= 0, \qquad v(0, x) = v_{0}, \qquad v_{t}(0, x) = v_{1}, \\
w_{tt} - \Delta w + \chi(\frac{u}{w}) |u|^{\frac{4}{d - 1}} u &= 0, \qquad w(0, x) = w_{0}, \qquad w_{t}(0, x) = w_{1}, \\
\endaligned
\end{equation}
$(v_{0}, v_{1})$ and $(w_{0}, w_{1})$ satisfy $(\ref{2.4})$--$(\ref{2.7})$, and $\chi \in C_{0}^{\infty}(\mathbb{R})$, $\chi(x) = 1$ for $|x| \leq 3$ and $\chi(x) = 0$ for $|x| > 6$.

\begin{theorem}[Small data result]\label{t7.1}
The initial value problem
\begin{equation}\label{7.5}
w_{tt} - \Delta w + \chi(\frac{u}{w}) |u|^{\frac{4}{d - 1}} u = 0, \qquad w(0, x) = w_{0}, \qquad w_{t}(0, x) = w_{1},
\end{equation}
is globally well-posed and scattering. Moreover,
\begin{equation}\label{7.6}
\| w \|_{L_{t,x}^{\frac{2(d + 1)}{d - 1}}(\mathbb{R} \times \mathbb{R}^{d})} + \| w \|_{L_{t}^{2} L_{x}^{\frac{2d}{d - 2}}(\mathbb{R} \times \mathbb{R}^{d})} \lesssim \epsilon,
\end{equation}
and for any $0 < \theta \leq 1$,
\begin{equation}\label{7.7}
\| |x|^{\frac{d - 2}{2} (1 - \theta)} w \|_{L_{t}^{2} L_{x}^{\frac{2d}{(d - 2) \theta}}(\mathbb{R} \times \mathbb{R}^{d})} \lesssim_{\theta} \epsilon.
\end{equation}
\end{theorem}
\begin{proof}
First note that by the approximation analysis in $(\ref{2.10})$--$(\ref{2.14})$ and persistence of regularity, $u$ and $w$ are smooth, so $\chi(\frac{u}{w^{(n)}})$ is well-defined.

Define the Picard iteration scheme
\begin{equation}\label{7.8}
w^{(0)}(t) = \cos(t \sqrt{-\Delta}) w_{0} + \frac{\sin(t \sqrt{-\Delta})}{\sqrt{-\Delta}} w_{1},
\end{equation}
and for $n \geq 1$,
\begin{equation}\label{7.9}
w^{(n)}(t) = w^{(0)}(t) - \int_{0}^{t} \frac{\sin((t - \tau) \sqrt{-\Delta}}{\sqrt{-\Delta}} \chi(\frac{u}{w^{(n - 1)}}) |u(\tau)|^{\frac{4}{d - 1}} u(\tau) d\tau.
\end{equation}
First, since $|\frac{u}{w^{(n - 1)}}| \leq 6$ on the support of $\chi(\frac{u}{w^{(n - 1)}})$,
\begin{equation}\label{7.10}
\| w^{(n)} \|_{L_{t,x}^{\frac{2(d + 1)}{d - 1}}(\mathbb{R} \times \mathbb{R}^{d})} \lesssim \epsilon + \| w^{(n - 1)} \|_{L_{t,x}^{\frac{2(d + 1)}{d - 1}}(\mathbb{R} \times \mathbb{R}^{d})}^{1 + \frac{4}{d - 1}}.
\end{equation}
Therefore, for $\epsilon > 0$ sufficiently small,
\begin{equation}\label{7.11}
\| w^{(n)} \|_{L_{t,x}^{\frac{2(d + 1)}{d - 1}}(\mathbb{R} \times \mathbb{R}^{d})} \lesssim \epsilon.
\end{equation}

Next,
\begin{equation}\label{7.12}
w^{(n + 1)}(t) - w^{(n)}(t) = \int_{0}^{t} \frac{\sin((t - \tau) \sqrt{-\Delta}}{\sqrt{-\Delta}} [\chi(\frac{u}{w^{(n)}}) - \chi(\frac{u}{w^{(n - 1)}})] |u(\tau)|^{\frac{4}{d - 1}} u(\tau) d\tau.
\end{equation}
We show
\begin{equation}\label{7.13}
[\chi(\frac{u}{w^{(n)}}) - \chi(\frac{u}{w^{(n - 1)}}] |u(\tau)|^{\frac{4}{d - 1}} u(\tau) \lesssim |w^{(n)} - w^{(n - 1)}|  (|w^{(n)}|^{\frac{4}{d - 1}} + |w^{(n - 1)}|^{\frac{4}{d - 1}}).
\end{equation}
Indeed, when $|w^{(n)} - w^{(n - 1)}| \gtrsim |w^{(n)}| + |w^{(n - 1)}|$,
\begin{equation}\label{7.14}
[\chi(\frac{u}{w^{(n)}}) - \chi(\frac{u}{w^{(n - 1)}})] |u(\tau)|^{\frac{4}{d - 1}} u(\tau) \lesssim |w^{(n)}|^{1 + \frac{4}{d - 1}} + |w^{(n - 1)}|^{1 + \frac{4}{d - 1}} \lesssim |w^{(n)} - w^{(n - 1)}| (|w^{(n)}|^{\frac{4}{d - 1}} + |w^{(n - 1)}|^{\frac{4}{d - 1}}).
\end{equation}
For $|w^{(n)} - w^{(n - 1)}| \ll |w^{(n)}| + |w^{(n - 1)}|$, by the fundamental theorem of calculus,
\begin{equation}\label{7.15}
\aligned
\chi(\frac{u}{w^{(n)}}) - \chi(\frac{u}{w^{(n - 1)}}) = \int_{0}^{1} \frac{d}{d\tau} \chi(\frac{u}{\tau w^{(n)} + (1 - \tau) w^{(n - 1)}}) d\tau \\ = -\int_{0}^{1} \chi'(\frac{u}{\tau w^{(n)} + (1 - \tau) w^{(n - 1)}}) \frac{u}{(\tau w^{(n)} + (1 - \tau) w^{(n - 1)})^{2}} \cdot (w^{(n)} - w^{(n - 1)}) d\tau.
\endaligned
\end{equation}
By the support properties of $\chi$,
\begin{equation}\label{7.16}
\chi'(\frac{u}{\tau w^{(n)} + (1 - \tau) w^{(n - 1)}}) \frac{u}{(\tau w^{(n)} + (1 - \tau) w^{(n - 1)})^{2}} \lesssim \frac{1}{u},
\end{equation}
so $(\ref{7.13})$ also holds. Therefore, by $(\ref{7.12})$,
\begin{equation}\label{7.17}
\| w^{(n + 1)} - w^{(n)} \|_{L_{t,x}^{\frac{2(d + 1)}{d - 1}}(\mathbb{R} \times \mathbb{R}^{d})} \lesssim \epsilon^{\frac{4}{d - 1}} \| w^{(n)} - w^{(n - 1)} \|_{L_{t,x}^{\frac{2(d + 1)}{d - 1}}(\mathbb{R} \times \mathbb{R}^{d})},
\end{equation}
which by the contraction mapping principle implies that $w^{(n)}$ converges in $L_{t,x}^{\frac{2(d + 1)}{d - 1}}$. By Theorem $\ref{t4.3}$ and Lemmas $\ref{l4.4}$ and $\ref{l4.5}$, $(\ref{7.6})$ and $(\ref{7.7})$ hold.
\end{proof}

\section{Scattering when $d > 4$}
Now we are ready to prove scattering when $d > 4$.
\begin{theorem}\label{t3.1}
If $d > 4$ and $u$ is a solution to the conformal wave equation,
\begin{equation}\label{3.1}
u_{tt} - \Delta u + |u|^{\frac{4}{d - 1}} u = 0, \qquad u(0,x) = u_{0} \in \dot{H}^{1/2}, \qquad u_{t}(0, x) = u_{1} \in \dot{H}^{-1/2}, \qquad u_{0}, u_{1} \qquad \text{radial},
\end{equation}
then $u$ is a global solution to $(\ref{3.1})$ and scatters, that is
\begin{equation}\label{3.2}
\| u \|_{L_{t,x}^{\frac{2(d + 1)}{d - 1}}(\mathbb{R} \times \mathbb{R}^{d})} \leq C(u_{0}, u_{1}) < \infty.
\end{equation}
\end{theorem}
\begin{proof}
Let
\begin{equation}\label{3.4}
F = (1 - \chi(\frac{u}{w})) |u|^{\frac{4}{d - 1}} u - |v|^{\frac{4}{d - 1}} v.
\end{equation}
As in $(\ref{2.9})$,
\begin{equation}\label{3.5}
\mathcal E(1) \lesssim R^{2} (\| v_{0} \|_{\dot{H}^{1/2}}^{2} + \| v_{1} \|_{\dot{H}^{-1/2}}^{2} + \| v_{0} \|_{\dot{H}^{1/2}}^{\frac{2(d + 1)}{d - 1}}).
\end{equation}
Also, as in $(\ref{2.16})$,
\begin{equation}\label{3.6}
\frac{d}{dt} \mathcal E(v) = -\langle (t + |x|) Lv + (d - 1)v, (t + |x|) F \rangle - \langle (t - |x|) \underline{L}v + (d - 1)v, (t - |x|) F \rangle.
\end{equation}
Since $1 - \chi(\frac{u}{w})$ is supported on the set $|u| \geq 3 |w|$, and therefore, $|v| \gtrsim |w|$ on the support of $(1 - \chi(\frac{u}{w}))$. Therefore,
\begin{equation}\label{3.8}
F \lesssim |v|^{\frac{d + 3}{d - 1}}.
\end{equation}
Also,
\begin{equation}\label{3.7}
\aligned
F = (1 - \chi(\frac{u}{w})) |u|^{\frac{4}{d - 1}} u - |v|^{\frac{4}{d - 1}} v = [|u|^{\frac{4}{d - 1}} u - |v|^{\frac{4}{d - 1}} v] - \chi(\frac{u}{w}) |u|^{\frac{4}{d - 1}} u \\ \lesssim |w| (|v|^{\frac{4}{d - 1}} + |w|^{\frac{4}{d - 1}}) + |w|^{\frac{d + 3}{d - 1}} \lesssim |w| |v|^{\frac{4}{d - 1}}.
\endaligned
\end{equation}

Next,
\begin{equation}\label{3.9}
\aligned
\| |x| F \|_{L_{x}^{2}(\mathbb{R}^{d})} \lesssim \| |x| w \|_{L_{x}^{\frac{2(d + 1)}{d - 3}}(\mathbb{R}^{d})} \| v \|_{L_{x}^{\frac{2(d + 1)}{d - 1}}(\mathbb{R}^{d})}^{\frac{4}{d - 1}} \lesssim \frac{\mathcal E(t)^{\frac{2}{d + 1}}}{t^{\frac{4}{d + 1}}} \| |x| w \|_{L_{x}^{\frac{2(d + 1)}{d - 3}}(\mathbb{R}^{d})}.
\endaligned
\end{equation}
Using $(\ref{4.17})$ and the radial Sobolev embedding theorem,
\begin{equation}\label{3.9.1}
\| |x| w \|_{L_{t,x}^{\frac{2(d + 1)}{d - 3}}(\mathbb{R} \times \mathbb{R}^{d})} \lesssim \| |x| w \|_{L_{t}^{2} L_{x}^{\frac{2d}{d - 4}}(\mathbb{R} \times \mathbb{R}^{d})}^{\frac{d - 3}{d + 1}} \| |x| w \|_{L_{t}^{\infty} L_{x}^{\frac{2d}{d - 3}}(\mathbb{R} \times \mathbb{R}^{d})}^{\frac{4}{d - 1}} \lesssim \epsilon.
\end{equation}

Plugging $(\ref{3.9})$ and $(\ref{3.9.1})$ into $(\ref{2.25})$ gives a similar bound. Next,
\begin{equation}\label{3.10}
\aligned
\| t F \|_{L^{2}(|x| \geq \delta |t|)} \lesssim \frac{1}{\delta} \| |x| F \|_{L^{2}},
\endaligned
\end{equation}
which we can also plug into $(\ref{3.9})$ and $(\ref{3.9.1})$.\medskip

Now define
\begin{equation}\label{3.11}
\mathcal R = \inf \{ (\frac{\mathcal E(t)}{t^{2}} + \frac{t^{\frac{2}{d + 1}} \mathcal E(t)^{\frac{d - 1}{d + 1}}}{t^{2}})^{-1}, \delta |t| \}.
\end{equation}
When $d > 5$ let $r = \frac{2(d - 1)}{d - 5}$, and then if $t' = \sup \{ 1, \delta^{1/2} t \}$,
\begin{equation}\label{3.12}
\aligned
t \int_{t'}^{t} \int_{\mathcal R \leq |x| \leq \delta |t|} |(\tau + |x|) Lv + (d - 1) v| |F| dx d\tau \\ \lesssim t \| (t + |x|) Lv + (d - 1) v \|_{L_{t}^{\infty} L_{x}^{2}}^{\frac{4}{d - 1}} \| |x|^{-1} \{ (t + |x|) Lv + (d - 1) v \} \|_{L_{t,x}^{2}(\mathcal R \leq |x| \leq \delta t)}^{\frac{d - 5}{d - 1}} \| \frac{1}{|x|^{3/2}} v \|_{L_{t,x}^{2}}^{\frac{4}{d - 1}} \| |x|^{\frac{d + 1}{d - 1}} w \|_{L_{t}^{2} L_{x}^{r}} \\
\lesssim t \mathcal E(t)^{\frac{2}{d - 1}} \| |x|^{-1} \{ (t + |x|) Lv + (d - 1) v \} \|_{L_{t,x}^{2}(\mathcal R \leq |x| \leq \delta t)}^{\frac{d - 5}{d - 1}} \| \frac{1}{|x|^{3/2}} v \|_{L_{t,x}^{2}}^{\frac{4}{d - 1}} \| |x|^{\frac{d + 1}{d - 1}} w \|_{L_{t}^{2} L_{x}^{r}}.
\endaligned
\end{equation}
\begin{remark}
Once again, the terms with $(t + |x|) Lv$ and $(t - |x|) \underline{L} v$ can be handled in exactly the same manner.
\end{remark}
Then by H{\"o}lder's inequality and $(\ref{3.11})$,
\begin{equation}\label{3.13}
\aligned
\lesssim t \mathcal E(t)^{\frac{2}{d - 1}} (\frac{\mathcal E(t)}{t^{2}} + \frac{t^{\frac{2}{d + 1}} \mathcal E(t)^{\frac{d - 1}{d + 1}}}{t^{2}})^{\frac{d - 5}{2(d - 1)}} \cdot (\sup_{\mathcal R \leq R \leq \frac{\delta t}{2}} R^{-1/2} t \| \nabla_{t,x} v \|_{L_{t,x}^{2}(R \leq |x| \leq 2R)} + t \| \frac{1}{|x|^{3/2}} v \|_{L_{t,x}^{2}(|x| \leq \delta t)})^{\frac{d - 5}{d - 1}} \\
\times (\int_{t'}^{t} \int_{|x| \leq \delta t} \frac{1}{|x|^{3}} v^{2} dx dt)^{\frac{2}{d - 1}} \| |x|^{\frac{d + 1}{d - 1}} w \|_{L_{t}^{2} L_{x}^{r}}.
\endaligned
\end{equation}
When $d = 5$,
\begin{equation}\label{3.15}
\aligned
t \int_{t'}^{t} \int_{\mathcal R \leq |x| \leq \delta |t|} |(\tau + |x|) Lv + (d - 1) v| |F| dx d\tau \\
\lesssim \| |x|^{-\frac{4}{9}} \{ |x| w \} \|_{L_{t}^{2} L_{x}^{6}(|x| \geq \mathcal R)} \| \frac{1}{|x|^{1/3}} v \|_{L_{t,x}^{3}} \| |x|^{-\frac{2}{3}} \{ (t + |x|) Lv + (d - 1) v \} \|_{L_{t,x}^{2}(|x| \geq \mathcal R)}^{1/3} \| (t + |x|) Lv + (d - 1) v \|_{L_{t}^{\infty} L_{x}^{2}}^{2/3} \\
\lesssim t \mathcal E(t)^{\frac{1}{3}} (\frac{\mathcal E(t)}{t^{2}} + \frac{t^{\frac{1}{3}} \mathcal E(t)^{\frac{2}{3}}}{t^{2}})^{\frac{1}{6}}(\int_{t'}^{t} \int_{|x| \leq \delta t} \frac{1}{|x|} |v|^{3} dx d\tau)^{1/3} \\ \times (\sup_{\mathcal R \leq R \leq \frac{\delta t}{2}} R^{-1/2} t \| \nabla_{t,x} v \|_{L_{t,x}^{2}(R \leq |x| \leq 2R)} + t \| \frac{1}{|x|^{3/2}} v \|_{L_{t,x}^{2}(|x| \leq \delta t)})^{\frac{1}{3}}  \| |x| w \|_{L_{t}^{2} L_{x}^{10}}.
\endaligned
\end{equation}

For $|x| \leq \mathcal R$, let
\begin{equation}\label{3.16.1}
s = \frac{d - 2}{2} - \frac{d - 2}{2} \frac{d - 1}{2} \alpha - \frac{d - 1}{2} \alpha, \qquad \frac{1}{p} = \frac{1}{2000d^{3}}, \qquad \alpha = \frac{1}{1000d^{3}}, \qquad \beta = \frac{1}{d - 3} \frac{(\frac{d - 3}{d - 1})}{(\frac{d - 3}{d - 1} + \alpha)},
\end{equation}
\begin{equation}\label{3.16}
\aligned
t \int_{t'}^{t} \int_{|x| \leq \mathcal R} |(t + |x|) Lv + (d - 1)v| |F| dx dt \\ \lesssim t \| (t + |x|) Lv + (d - 1) v \|_{L_{t}^{\infty} L_{x}^{2}}^{\frac{2}{d - 1} - \alpha} \| |x|^{-\frac{1}{2} + \beta} \{ (t + |x|) Lv + (d - 1) v \} \|_{L_{t,x}^{2}(|x| \leq \mathcal R)}^{\frac{d - 3}{d - 1} + \alpha} \\ \times \| \frac{1}{|x|^{3/2}} v \|_{L_{t,x}^{2}}^{1 - \alpha} \| |x|^{\frac{d - 2}{2}} v \|_{L_{t,x}^{\infty}}^{\frac{2}{d - 1} + \alpha}  \| |x|^{s} w \|_{L_{t}^{2} L_{x}^{p}}^{\frac{2}{d - 1}}  \\
\lesssim t \mathcal E(t)^{\frac{1}{d - 1} - \frac{\alpha}{2}} \mathcal R^{\frac{1}{d - 1}} t^{\frac{d - 3}{d - 1} + \alpha}(\sup_{0 < R \leq \delta t} R^{-1/2} \| \nabla_{t,x} v \|_{L_{t,x}^{2}(|x| \leq R)} + \| \frac{1}{|x|^{3/2}} v \|_{L_{t,x}^{2}(|x| \leq \delta t)})^{\frac{d - 3}{(d - 1)} + \alpha} \\
\times \| \frac{1}{|x|^{3/2}} v \|_{L_{t,x}^{2}(|x| \leq \delta t)}^{1 - \alpha} \| \chi(\frac{r}{\delta t}) v \|_{\dot{H}^{1}}^{\frac{2}{d - 1} + \alpha} \| |x|^{s} w \|_{L_{t}^{2} L_{x}^{p}}^{\frac{2}{d - 1}}.
\endaligned
\end{equation}

Now, similar to Proposition $\ref{p2.4}$,
\begin{proposition}\label{p3.2}
For $d > 4$, if $T > 1$, $T' = \sup \{ 1, \delta^{1/2} T \}$,
\begin{equation}\label{3.17}
\aligned
\sup_{R > 0} R^{-1} \int_{T'}^{T} \int_{|x| \leq R} \chi(\frac{x}{\delta T}) [|\nabla v|^{2} + v_{t}^{2}] dx dt \\ + \int_{T'}^{T} \int \chi(\frac{x}{\delta T}) [\frac{1}{|x|^{3}} v^{2} + \frac{1}{|x|} |v|^{\frac{2(d + 1)}{d - 1}}] dx dt \lesssim_{\delta} \frac{\mathcal E(T)}{T^{2}} + \frac{T^{\frac{2}{d + 1}} \mathcal E(T)^{\frac{d - 1}{d + 1}}}{T^{2}}.
\endaligned
\end{equation}
\end{proposition}
\begin{proof}
Following $(\ref{2.31})$ and $(\ref{2.32})$, for $t \in [\delta^{1/2} t, t]$,
\begin{equation}\label{3.18}
\| \chi(\frac{x}{\delta T}) \nabla_{t,x} v \|_{L^{2}}^{2} \lesssim \frac{\mathcal E(T)}{T^{2}} + \frac{1}{T^{2}} \| v \|_{L^{2}(|x| \leq 2 \delta T)}^{2},
\end{equation}
and by H{\"o}lder's inequality,
\begin{equation}\label{3.19}
\frac{1}{T^{2}} \| v \|_{L^{2}(|x| \leq 2 \delta T)}^{2} \lesssim T^{-\frac{2}{d + 1}} \| v \|_{L_{x}^{\frac{2(d + 1)}{d - 1}}}^{2} \lesssim T^{-\frac{2}{d + 1}} \frac{\mathcal E(T)^{\frac{d - 1}{d + 1}}}{t^{\frac{2(d - 1)}{d + 1}}} \lesssim \frac{T^{\frac{2}{d + 1}} \mathcal E(T)^{\frac{d - 1}{d + 1}}}{T^{2}}.
\end{equation}
By Propositions $\ref{p5.3}$ and $\ref{p5.4}$, it only remains to handle the error terms arising from $F$, where $F$ satisfies $(\ref{3.8})$ and $(\ref{3.7})$. We can do this using $(\ref{3.12})$, $(\ref{3.13})$, and $(\ref{3.15})$ combined with the analysis in $(\ref{2.65})$--$(\ref{2.72})$ applied to the $d > 4$ case.
\end{proof}

Plugging in $(\ref{3.17})$ to $(\ref{3.12})$--$(\ref{3.15})$, if $t' = \sup \{ 1, \delta^{1/2} t \}$,

\begin{equation}\label{3.20}
\aligned
\mathcal E(t) \lesssim \mathcal E(t') + \int_{t'}^{t} \frac{\mathcal E(\tau)^{\frac{2}{d + 1} + \frac{1}{2}}}{\tau^{\frac{4}{d + 1}}} \| |x| w \|_{L_{x}^{\frac{2(d + 1)}{d - 3}}(\mathbb{R}^{d})} d\tau \\ + C(\delta) t \mathcal E(t)^{\frac{1}{d - 1} - \frac{\alpha}{2}} t^{\frac{d - 3}{d - 1} + \alpha} (\frac{\mathcal E(t)}{t^{2}} + \frac{t^{\frac{2}{d + 1}} \mathcal E(t)^{\frac{d - 1}{d + 1}}}{t^{2}})^{\frac{d - 2}{d - 1} + \frac{\alpha}{2}} \| |x|^{s} w \|_{L_{t}^{2} L_{x}^{p}}^{\frac{2}{d - 1}} \\
+ C(\delta) t^{4/3} \mathcal E(t)^{1/3} (\frac{\mathcal E(t)}{t^{2}} + \frac{t^{1/3} \mathcal E(t)^{2/3}}{t^{2}})^{2/3} \| |x| w \|_{L_{t}^{2} L_{x}^{10}}, \qquad \text{if } d = 5, \\
+ C(\delta) t t^{\frac{d - 5}{d - 1}} \mathcal E(t)^{\frac{2}{d - 1}} (\frac{\mathcal E(t)}{t^{2}} + \frac{t^{\frac{2}{d + 1}} \mathcal E(t)^{\frac{d - 1}{d + 1}}}{t^{2}}) \| |x|^{\frac{d + 1}{d - 1}} w \|_{L_{t}^{2} L_{x}^{\frac{2(d - 1)}{d - 5}}}, \qquad \text{if} \qquad d > 5,
\endaligned
\end{equation}
where $s, p, \alpha, \beta$ are given by $(\ref{3.16.1})$. Following the computations in $(\ref{2.23})$--$(\ref{2.25})$ and $(\ref{2.47})$--$(\ref{2.67})$, we obtain a uniform bound on
\begin{equation}\label{3.21}
\int_{1}^{\infty} \frac{\mathcal E(t)}{t^{2}} dt < \infty,
\end{equation}
which proves the theorem.

\end{proof}

\section{Profile decomposition}
\begin{proof}[Proof of Theorem $\ref{t1.1}$]
In light of Theorems $\ref{t2.1}$ and $\ref{t3.1}$, to prove Theorem $\ref{t1.1}$, it suffices to prove that if $(u_{n}^{0}, u_{n}^{1})$ is a sequence of initial data satisfying
\begin{equation}\label{6.1}
\| u_{0}^{n} \|_{\dot{H}^{1/2}} + \| u_{1}^{n} \|_{\dot{H}^{-1/2}} \leq A < \infty,
\end{equation}
then
\begin{equation}\label{6.2}
\sup_{n} \| u^{n} \|_{L_{t,x}^{\frac{2(d + 1)}{d - 1}}(\mathbb{R} \times \mathbb{R}^{d})} < \infty.
\end{equation}
is uniformly bounded, where $u^{n}$ is the solution to $(\ref{1.1})$ with initial data $(u_{0}^{n}, u_{1}^{n})$.

Indeed, since $\dot{H}^{1/2} \times \dot{H}^{-1/2}$ is separable, we can take a dense sequence $(u_{0}^{n}, u_{1}^{n})$ in $(\ref{6.1})$. Passing to a subsequence where $(\ref{6.2})$ is increasing, it is enough to show that Theorems $\ref{t2.1}$, $\ref{t3.1}$, and a standard profile decomposition argument imply that $(\ref{6.2})$ is uniformly bounded for any $A$. By standard perturbative arguments, Theorem $\ref{t1.1}$ follows.
\begin{remark}
Observe that this argument does not give any idea how the function on the right hand side of $(\ref{1.3})$ depends on $A$.
\end{remark}
The argument proving $(\ref{6.2})$ is identical to the argument in \cite{dodson2018global} for the cubic wave equation, $(\ref{1.1})$ with $d = 3$, and uses the profile decomposition in \cite{ramos2012refinement}.

\begin{remark}
It is useful to use the notation $S(t)(f, g)$, which denotes the solution to the free wave equation with initial data $(f, g)$,
\begin{equation}\label{6.9}
S(t)(f, g) = \cos(t \sqrt{-\Delta}) f + \frac{\sin(t \sqrt{-\Delta})}{\sqrt{-\Delta}} g.
\end{equation}
\end{remark}

\begin{theorem}[Profile decomposition]\label{t6.1}
Suppose that there is a uniformly bounded, radially symmetric sequence
\begin{equation}\label{6.3}
\| u_{0}^{n} \|_{\dot{H}^{1/2}(\mathbb{R}^{d})} + \| u_{1}^{n} \|_{\dot{H}^{-1/2}(\mathbb{R}^{d})} \leq A < \infty.
\end{equation}
Then there exists a subsequence, also denoted $(u_{0}^{n}, u_{1}^{n}) \subset \dot{H}^{1/2} \times \dot{H}^{-1/2}$ such that for any $N < \infty$,
\begin{equation}\label{6.4}
S(t)(u_{0}^{n}, u_{1}^{n}) = \sum_{j = 1}^{N} \Gamma_{n}^{j} S(t)(\phi_{0}^{j}, \phi_{1}^{j}) + S(t)(R_{0, n}^{N}, R_{1,n}^{N}),
\end{equation}
with
\begin{equation}\label{6.5}
\lim_{N \rightarrow \infty} \limsup_{n \rightarrow \infty} \| S(t)(R_{0,n}^{N}, R_{1,n}^{N}) \|_{L_{t,x}^{\frac{2(d + 1)}{d - 1}}(\mathbb{R} \times \mathbb{R}^{d})} = 0.
\end{equation}
Here, $\Gamma_{n}^{j} = (\lambda_{n}^{j}, t_{n}^{j})$ belongs to the group $(0, \infty) \times \mathbb{R}$, which acts by
\begin{equation}\label{6.6}
\Gamma_{n}^{j} F(t,x) = (\lambda_{n}^{j})^{\frac{d - 1}{2}} F(\lambda_{n}^{j} (t - t_{n}^{j}), \lambda_{n}^{j} x).
\end{equation}
The $\Gamma_{n}^{j}$ are pairwise orthogonal, that is, for every $j \neq k$,
\begin{equation}\label{6.7}
\lim_{n \rightarrow \infty} \frac{\lambda_{n}^{j}}{\lambda_{n}^{k}} + \frac{\lambda_{n}^{k}}{\lambda_{n}^{j}} + (\lambda_{n}^{j})^{1/2} (\lambda_{n}^{k})^{1/2} |t_{n}^{j} - t_{n}^{k}| = \infty.
\end{equation}
Furthermore, for every $N \geq 1$,
\begin{equation}\label{6.8}
\aligned
\| (u_{0, n}, u_{1, n}) \|_{\dot{H}^{1/2} \times \dot{H}^{-1/2}}^{2} = \sum_{j = 1}^{N} \| (\phi_{0}^{j}, \phi_{0}^{k}) \|_{\dot{H}^{1/2} \times \dot{H}^{-1/2}}^{2} \\ + \| (R_{0, n}^{N}, R_{1, n}^{N}) \|_{\dot{H}^{1/2} \times \dot{H}^{-1/2}}^{2} + o_{n}(1).
\endaligned
\end{equation}
\end{theorem}

In the course of proving Theorem $\ref{t6.1}$, \cite{ramos2012refinement} proved
\begin{equation}\label{6.10}
S(\lambda_{n}^{j} t_{n}^{j})(\frac{1}{(\lambda_{n}^{j})^{\frac{d - 1}{2}}} u_{0}^{n}(\frac{x}{\lambda_{n}^{j}}), \frac{1}{(\lambda_{n}^{j})^{\frac{d + 1}{2}}} u_{1}^{n}(\frac{x}{\lambda_{n}^{j}})) \rightharpoonup \phi_{0}^{j}(x),
\end{equation}
weakly in $\dot{H}^{1/2}(\mathbb{R}^{d})$, and
\begin{equation}\label{6.11}
\partial_{t}S(t + \lambda_{n}^{j} t_{n}^{j})(\frac{1}{(\lambda_{n}^{j})^{\frac{d - 1}{2}}} u_{0}^{n}(\frac{x}{\lambda_{n}^{j}}), \frac{1}{(\lambda_{n}^{j})^{\frac{d + 1}{2}}} u_{1}^{n}(\frac{x}{\lambda_{n}^{j}}))|_{t = 0} \rightharpoonup \phi_{1}^{j}(x)
\end{equation}
weakly in $\dot{H}^{-1/2}(\mathbb{R}^{d})$.\medskip

Suppose that for some $j$, $\lambda_{n}^{j} t_{n}^{j}$ is uniformly bounded. Then after passing to a subsequence, $\lambda_{n}^{j} t_{n}^{j}$ converges to some $t^{j}$. Changing $(\phi_{0}^{j}, \phi_{1}^{j})$ to $(S(-t^{j})(\phi_{0}^{j}, \phi_{1}^{j}), \partial_{t} S(t - t^{j})(\phi_{0}^{j}, \phi_{1}^{j})|_{t = 0})$ and absorbing the error into $(R_{0, n}^{N}, R_{1, n}^{N})$,
\begin{equation}\label{6.12}
(\frac{1}{(\lambda_{n}^{j})^{\frac{d - 1}{2}}} u_{0}^{n}(\frac{x}{\lambda_{n}^{j}}), \frac{1}{(\lambda_{n}^{j})^{\frac{d + 1}{2}}} u_{1}^{n}(\frac{x}{\lambda_{n}^{j}})) \rightharpoonup \phi_{0}^{j}(x), \qquad \text{weakly in} \qquad \dot{H}^{1/2}
\end{equation}
and
\begin{equation}\label{6.13}
\partial_{t}S(t)(\frac{1}{(\lambda_{n}^{j})^{\frac{d - 1}{2}}} u_{0}^{n}(\frac{x}{\lambda_{n}^{j}}), \frac{1}{(\lambda_{n}^{j})^{\frac{d + 1}{2}}} u_{1}^{n}(\frac{x}{\lambda_{n}^{j}}))|_{t = 0} \rightharpoonup \phi_{1}^{j}(x), \qquad \text{weakly in} \qquad \dot{H}^{-1/2}.
\end{equation}
If $u^{(j)}$ is the solution to $(\ref{1.1})$ with initial data $(\phi_{0}^{j}, \phi_{1}^{j})$, then
\begin{equation}\label{6.14}
\| u^{(j)} \|_{L_{t,x}^{\frac{2(d + 1)}{d - 1}}(\mathbb{R} \times \mathbb{R}^{d})} \leq M_{j}.
\end{equation}

Next, suppose that after passing to a subsequence, $\lambda_{n}^{j} t_{n}^{j} \nearrow +\infty$. In this case, for any $(\phi_{0}^{j}, \phi_{1}^{j}) \in \dot{H}^{1/2} \times \dot{H}^{-1/2}$, there exists a solution $u^{(j)}$ to $(\ref{1.1})$ that is globally well-posed and scattering, and furthermore, that $u$ scatters to $S(t)(\phi_{0}^{j}, \phi_{1}^{j})$ as $t \searrow -\infty$.
\begin{equation}\label{6.15}
\lim_{t \rightarrow -\infty} \| u^{(j)}(t) - S(t)(\phi_{0}^{j}, \phi_{1}^{j}) \|_{\dot{H}^{1/2} \times \dot{H}^{-1/2}} = 0.
\end{equation}
Indeed, by Strichartz estimates, the dominated convergence theorem, and the small data arguments in Theorem $\ref{t4.2}$, for some $T_{j} < \infty$ sufficiently large, $(\ref{1.1})$ has a solution $u$ on $(-\infty, -T]$ such that
\begin{equation}\label{6.16}
\| u^{(j)} \|_{L_{t,x}^{\frac{2(d + 1)}{d - 1}}((-\infty, -T_{j}] \times \mathbb{R}^{d})} \lesssim \epsilon_{0}(d), \qquad (u^{(j)}(-T_{j}, x), u_{t}^{(j)}(-T_{j}, x)) = S(-T_{j})(\phi_{0}^{j}, \phi_{1}^{j}),
\end{equation}
where $\epsilon_{0}(d) > 0$ is sufficiently small. Also by Strichartz estimates and small data arguments,
\begin{equation}\label{6.17}
\lim_{t \rightarrow +\infty} \| S(t)(u^{(j)}(-t), u_{t}^{(j)}(-t)) - (\phi_{0}, \phi_{1}) \|_{\dot{H}^{1/2} \times \dot{H}^{-1/2}} \lesssim \epsilon^{\frac{2(d + 1)}{d - 1}}.
\end{equation}
Then by the inverse function theorem, there exists some $(u_{0}^{(j)}(-T_{j}), u_{1}^{(j)}(-T_{j}))$ such that $(\ref{1.1})$ has a solution that scatters backward in time to $S(t)(\phi_{0}^{j}, \phi_{1}^{j})$. Since $u_{0}^{(j)}(-T_{j}) \in \dot{H}^{1/2}$ and $u_{1}^{(j)}(-T_{j}) \in \dot{H}^{-1/2}$, $(\ref{1.1})$ has a solution that scatters forward and backward in time,
\begin{equation}\label{6.18}
\| u^{(j)} \|_{L_{t,x}^{\frac{2(d + 1)}{d - 1}}(\mathbb{R} \times \mathbb{R}^{d})} \leq M_{j} < \infty,
\end{equation}
and $u^{(j)}(-T_{j}, x) = u_{0}^{(j)}(-T_{j}, x)$, $u_{t}^{(j)}(-T_{j}, x) = u_{1}^{(j)}(-T_{j}, x)$. Therefore,
\begin{equation}\label{6.19}
S(-t_{n}^{j})((\lambda_{n}^{j})^{\frac{d - 1}{2}} \phi_{0}^{j}(\lambda_{n}^{j} x), (\lambda_{n}^{j})^{\frac{d + 1}{2}} \phi_{1}^{j}(\lambda_{n}^{j} x))
\end{equation}
converges strongly to
\begin{equation}\label{6.20}
((\lambda_{n}^{j})^{\frac{d - 1}{2}} u^{(j)}(-\lambda_{n}^{j} t_{n}^{j}, \lambda_{n}^{j} x), (\lambda_{n}^{j})^{\frac{d + 1}{2}} u_{t}^{(j)}(-\lambda_{n}^{j} t_{n}^{j}, \lambda_{n}^{j} x))
\end{equation}
in $\dot{H}^{1/2} \times \dot{H}^{-1/2}$, where $u^{j}$ is the solution to $(\ref{1.1})$ that scatters backward in time to $S(t)(\phi_{0}^{j}, \phi_{1}^{j})$, and the remainder may be absorbed into $(R_{0, n}^{N}, R_{1, n}^{N})$. 


The proof for $\lambda_{n}^{j} t_{n}^{j} \searrow -\infty$ is similar.\medskip

By $(\ref{6.8})$, there are only finitely many $j$, say $J$, such that $\| \phi_{0}^{j} \|_{\dot{H}^{1/2}} + \| \phi_{1}^{j} \|_{\dot{H}^{-1/2}} > \epsilon_{0}(d)$. For all other $j$, small data arguments imply
\begin{equation}\label{6.22}
\| u^{(j)} \|_{L_{t,x}^{\frac{2(d + 1)}{d - 1}}(\mathbb{R} \times \mathbb{R}^{d})} \lesssim \| \phi_{0}^{j} \|_{\dot{H}^{1/2}} + \| \phi_{1}^{j} \|_{\dot{H}^{-1/2}}.
\end{equation}
Then by the decoupling property $(\ref{6.7})$, $(\ref{6.8})$, $(\ref{6.12})$, $(\ref{6.22})$, and Theorem $\ref{t4.2}$,
\begin{equation}\label{6.23}
\limsup_{n \nearrow \infty} \| u^{n} \|_{L_{t,x}^{\frac{2(d + 1)}{d - 1}}(\mathbb{R} \times \mathbb{R}^{d})}^{2} \lesssim \sum_{j} \| u^{(j)} \|_{L_{t,x}^{\frac{2(d + 1)}{d - 1}}(\mathbb{R} \times \mathbb{R}^{d})}^{2} \lesssim \sum_{j = 1}^{J} M_{j}^{2} + A^{2} < \infty.
\end{equation}
This proves that $(\ref{7.2})$ holds.
\end{proof}

\section*{Acknowledgements}
During the time of writing this paper, the author was partially supported by NSF Grant DMS-$1764358$. The author is also grateful to Andrew Lawrie and Walter Strauss for some helpful conversations at MIT on subcritical nonlinear wave equations.

\bibliography{biblio}
\bibliographystyle{alpha}

\end{document}